\documentclass[12pt,oneside,english]{amsart}
\usepackage[T1]{fontenc}
\usepackage[latin9]{inputenc}
\usepackage[a4paper]{geometry}
\usepackage{babel}
\usepackage{prettyref}
\usepackage{mathrsfs}
\usepackage{mathtools}
\usepackage{amstext}
\usepackage{amsthm}
\usepackage{amssymb}
\usepackage[unicode=true,
 bookmarks=false,
 breaklinks=false,pdfborder={0 0 1},backref=section,colorlinks=false]
 {hyperref}

\makeatletter
\numberwithin{equation}{section}
\numberwithin{figure}{section}
\theoremstyle{plain}
\newtheorem{thm}{\protect\theoremname}
\theoremstyle{plain}
\newtheorem{prop}[thm]{\protect\propositionname}
\theoremstyle{remark}
\newtheorem{rem}[thm]{\protect\remarkname}
\theoremstyle{plain}
\newtheorem{cor}[thm]{\protect\corollaryname}
\theoremstyle{definition}
\newtheorem{example}[thm]{\protect\examplename}

\usepackage{babel}

\makeatother

\providecommand{\corollaryname}{Corollary}
\providecommand{\examplename}{Example}
\providecommand{\propositionname}{Proposition}
\providecommand{\remarkname}{Remark}
\providecommand{\theoremname}{Theorem}

\begin{document}
\title[Bochner Laplacian and semipositive Bergman kernel]{Bochner Laplacian and Bergman kernel expansion of semipositive line
bundles on a Riemann surface}
\author{George Marinescu}
\address{Universit{\"a}t zu K{\"o}ln, Mathematisches Institut, 
Weyertal 86-90, 50931 K{\"o}ln, 
Deutschland    \newline
\mbox{\quad}\,Institute of Mathematics `Simion Stoilow', 
Romanian Academy, Bucharest, Romania}
\email{gmarines@math.uni-koeln.de}
\author{Nikhil Savale}
\thanks{G. M. and N. S. 
are partially supported by the DFG funded project SFB/TRR 191 
`Symplectic Structures in Geometry, Algebra and Dynamics' 
(Project-ID 281071066-TRR 191)
and the ANR-DFG project `Quantization, 
Singularities, and Holomorphic Dynamics' (Project-ID 490843120).
}
\address{Universität zu Köln, Mathematisches Institut, Weyertal 86-90, 50931
Köln, Germany}
\email{nsavale@math.uni-koeln.de}
\begin{abstract}
We generalize the results of Montgomery \cite{Montgomery-95paper}
for the Bochner Laplacian on high tensor powers of a line bundle.
When specialized to Riemann surfaces, this leads to the Bergman kernel
expansion for semipositive line bundles whose curvature vanishes at
finite order. The proof exploits the relation of the Bochner Laplacian
on tensor powers with the sub-Riemannian (sR) Laplacian.
\end{abstract}

\maketitle

\section{Introduction}

In \cite{Montgomery-95paper} Montgomery studied the spectrum, and
in particular the smallest eigenvalue, of the Bochner (magnetic) Laplacian
on the tensor powers $L^{k}\coloneqq L^{\otimes k}$ of a Hermitian
line bundle $L$. He assumed that the underlying manifold is a Riemann
surface and that the curvature of the line bundle vanishes transversally
along a curve. The problem goes back at least to the work Simon et
al. \cite{Avron-Herbst-Simon78,Simon-83} and Guillemin-Uribe \cite{Guillemin-Uribe88}
among others, who assumed the curvature is symplectic. The problem
has since also been actively explored under different assumptions
on the curvature. The first theorem in this article proves the most
general such leading asymptotic for the smallest eigenvalue of the
Bochner Laplacian on tensor powers.

The holomorphic analog of the above is the study of the Bergman kernel
of a holomorphic line bundle $L$ on a complex manifold. The Bergman
kernel is the Schwartz kernel of the projector from smooth sections
of $L$ onto holomorphic ones. The analysis of the Bergman kernel
and holomorphic sections associated to tensor powers has important
applications in complex geometry (see \cite{Demailly-published-text,Ma-Marinescu}).
When $L$ is positive, the leading asymptotic for the Bergman kernel
along the diagonal was first proved in \cite{Tian90} and later improved
to a full expansion in \cite{Catlin97-Bergmankernel,Zelditch98-Bergmankernel}
using the Szeg\H{o} kernel parametrix of \cite{Boutet-Sjostrand76}.
Subsequently, a different geometric method for the expansion was developed
in \cite{Dai-Liu-Ma2006-Bergman-kernel,Ma-Marinescu,MaMarinescu2008}
inspired by the analytic localization technique of \cite{Bismut-Lebeau91}.
The problem of the expansion for semipositive line bundles is largely
open. A second objective of this article is to give the first proof
of the Bergman kernel expansion at vanishing points of the curvature.
We achieve this for a semipositive line bundle $L$ on a Riemann surface. 

\subsection{Statement of the main results}

We now state our results more precisely. Let $Y^{n-1}$ be a compact
Riemannian manifold of dimension $n-1$ with complex Hermitian line
bundle $\left(L,h^{L}\right)$ and vector bundle $\left(F,h^{F}\right)$.
We equip these with unitary connections $\nabla^{L}$, $\nabla^{F}$
to obtain the Bochner Laplacian 
\begin{equation}
\Delta_{k}\coloneqq\left(\nabla^{F\otimes L^{k}}\right)^{*}\nabla^{F\otimes L^{k}}:C^{\infty}\left(Y;F\otimes L^{k}\right)\rightarrow C^{\infty}\left(Y;F\otimes L^{k}\right),\quad k\in\mathbb{N},\label{e:BL}
\end{equation}
on tensor powers $F\otimes L^{k}$, where the adjoint is taken with
respect to the natural $L^{2}$ metric. As the above is elliptic,
self-adjoint and positive, one has a complete orthonormal basis $\left\{ \psi_{j}^{k}\right\} _{j=1}^{\infty}$
of $L^{2}\left(Y;F\otimes L^{k}\right)$ consisting of its eigenvectors
$\Delta_{k}\psi_{j}^{k}=\lambda_{j}\left(k\right)\psi_{j}^{k}$, with
eigenvalues $0\leq\lambda_{0}\leq\lambda_{1}\ldots$. Denote by $R^{L}=\left(\nabla^{L}\right)^{2}\in\Omega^{2}\left(Y;i\mathbb{R}\right)$
the purely imaginary curvature form of the unitary connection $\nabla^{L}$.
The order of vanishing of $R^{L}$ at a point $y\in Y$ is now defined\footnote{The reason for this normalization, besides a simplification of resulting
formulas, is the significance of $r_{y}$ as the degree of nonholonomy
of a relevant sR distribution (see Proposition \prettyref{prop: degree of nonholonomy is ord van.}).} 
\begin{equation}
r_{y}-2=\textrm{ord}_{y}\left(R^{L}\right)\coloneqq\min\left\{ l|J^{l}\left(\Lambda^{2}T^{*}Y\right)\ni j_{y}^{l}R^{L}\neq0\right\} ,\quad r_{y}\geq2,\label{eq:order of vanishing curv.}
\end{equation}
where $j^{l}R^{L}$ denotes the $l$th jet of the curvature. We shall
assume that this order of vanishing is finite at all points of the
manifold i.e. 
\begin{equation}
r\coloneqq\max_{y\in Y}r_{y}<\infty.\label{eq: curv vanishes finite order}
\end{equation}
The function $y\mapsto r_{y}$ being upper semi-continuous, gives
a decomposition $Y=\bigcup_{j=2}^{r}Y_{j}$ of the manifold via 
\begin{align}
Y_{j}\coloneqq\left\{ y\in Y|r_{y}=j\right\} \quad\text{with each }\quad & Y_{\leq j}\coloneqq\bigcup_{j'=2}^{j}Y_{j'}\label{eq:Yr def}
\end{align}
being open. Our first theorem is now the following. 
\begin{thm}
\label{thm:First eigenvalue Bochner} Let $(L,h^{L})\to\left(Y,g^{TY}\right)$,
$(F,h^{F})\to\left(Y,g^{TY}\right)$ be Hermitian line and vector
bundles on a compact Riemannian manifold with unitary connections
$\nabla^{L}$, $\nabla^{F}$ . Assuming that the curvature $R^{L}$
vanishes to finite order at all points, with maximal order $r$ \prettyref{eq: curv vanishes finite order},
the first eigenvalue $\lambda_{0}\left(k\right)$ of the Bochner Laplacian
satisfies 
\begin{equation}
\lambda_{0}\left(k\right)\sim Ck^{2/r},\quad\textrm{as }k\rightarrow\infty,\label{eq: estimate first eigenvalue}
\end{equation}
for some positive constant $C$. Moreover, the first eigenfunction
concentrates on $Y_{r}$: 
\begin{equation}
\left|\psi_{0}^{k}\left(y\right)\right|=O\left(k^{-\infty}\right);\;y\in Y_{\leq r-1}.\label{eq: estimate first eigenfunction}
\end{equation}
\end{thm}

The leading constant above \prettyref{eq: estimate first eigenvalue}
can be identified 
\begin{equation}
C=\inf_{y\in Y_{r}}\lambda_{0}\left(\Delta_{y}\right)\label{eq: upper constant}
\end{equation}
in terms of the bottom of the spectrum of certain model Laplacians
$\Delta_{y}\coloneqq\Delta_{g_{y}^{TY},j^{r-2}R_{y}^{L}}$, depending
on the metric $g_{y}^{TY}$ and first non-vanishing jet tensors $j^{r-2}R_{y}^{L}$,
defined on the tangent space $T_{y}Y$ at each $y\in Y$ (see \prettyref{sec:Model-operators}).
The first case of the above theorem is $r=2$, when the curvature
$R^{L}$ is non-vanishing, and can be found in \cite{Helffer-Mohamed96}.
Here the model Laplacian is a harmonic oscillator. The bottom of its
spectrum is explicitly given $\lambda_{0}\left(\Delta_{y}\right)=\frac{1}{2}\textrm{tr }\sqrt{-J_{y}^{2}}$
in terms of the endomorphism $J_{y}:T_{y}Y\rightarrow T_{y}Y$, defined
by the equation $g^{TY}\left(.,J_{y}.\right)=R^{L}\left(.,.\right)$.
In \cite{Montgomery-95paper} a particular case of $r=3$, with $Y$
a Riemann surface, is considered. It is surpising that the general
case, despite being attempted, has been missed therein and in several
references since then.

Without further hypotheses, the structure of the locus $Y_{r}$ may
be quite general; locally any closed subset of a hypersurface (see
Section \prettyref{subsec:Structure-of Yr} below). To obtain further
information on the small eigenvalues, we introduce additional assumptions.
First, we assume $Y_{r}=\bigcup_{j=1}^{N}Y_{r,j}$ to be a union of
embedded submanifolds of dimensions $d_{j}\coloneqq\textrm{dim}\left(Y_{r,j}\right)$.
At points $y\in Y_{r}$, the first non-vanishing jet of the curvature
$j_{y}^{r-2}R^{L}\in S^{r-2}T_{y}^{*}Y\otimes\Lambda^{2}T_{y}^{*}Y$
may be thought of as an element of the product with the $\left(r-2\right)$th
symmetric power. We say that the curvature $R^{L}$ vanishes \textit{non-degenerately}
along $Y_{r}$ if the following implication holds
\begin{equation}
i_{v}^{s}\left(j_{y}^{r-2}R^{L}\right)=0,\,\forall s\leq r-2\,\implies v\in T_{y}Y_{r},\label{eq:non-degeneracy assumption}
\end{equation}
where $i^{s}$ above denotes the $s$-fold contraction of the symmetric
part of $j_{y}^{r-2}R^{L}$. In Remark \prettyref{rem:(Non-degeneracy-hypothesis)}
below we note the following less invariant definition of the non-degeneracy
condition \prettyref{eq:non-degeneracy assumption} above in local
coordinates: it is equivalent to assuming that the leading order part
$R_{0}^{L}$ in the Taylor expansion of the curvature at $y\in Y_{r}$
locally defines the same locus $Y_{r}=Y_{r}^{0}\coloneqq\left\{ y\in Y|\textrm{ord}_{y}\left(R_{0}^{L}\right)=r-2\right\} $
as \prettyref{eq:Yr def}. In Remark \prettyref{rem:(Non-degeneracy-hypothesis)}
we also note that our non-degeneracy condition \prettyref{eq:non-degeneracy assumption}
on the curvature is less restrictive than the assumption of $Y_{r}$
being a 'magnetic well' for the curvature $R^{L}$ that appears in
earlier works \cite{Helffer-Mohamed96}.

Now set $d_{j}^{\textrm{max}}\coloneqq\max\left\{ d_{j}\right\} _{j=1}^{N}$
and let $NY_{r,j}\coloneqq TY_{r,j}^{\perp}\subset TY$ denote the
normal bundle of each $Y_{r,j}$. Note that there is a natural density
on each $NY_{r,j}$ coming from the metric. Denote by $\chi_{\left[c_{1},c_{2}\right]}$
the characteristic function for $\left[c_{1},c_{2}\right]$. In Section
\prettyref{subsec:Weyl-laws} we show that under the non-degeneracy
hypothesis \prettyref{eq:non-degeneracy assumption}, the Schwartz
kernel of the model Laplacian on the tangent space $\chi_{\left[c_{1},c_{2}\right]}\left(\Delta_{y}\right)\left(v,v\right)=O\left(\left|v\right|^{-\infty}\right),\:v\in NY_{r,j,y},$
is rapidly decaying, and thus integrable, in the normal directions. 

Our next result is on the asymptotics for the Weyl counting function
$N\left[c_{1}k^{2/r},c_{2}k^{2/r}\right]$ for the number of eigenvalues
of $\Delta_{k}$ in the given interval. 
\begin{thm}
\label{thm: Weyl law =00003D000026 expansion} Let $(L,h^{L})\to\left(Y,g^{TY}\right)$,
$(F,h^{F})\to\left(Y,g^{TY}\right)$ be Hermitian line and vector
bundles on a compact Riemannian manifold with unitary connections
$\nabla^{L}$, $\nabla^{F}$ . Assuming $Y_{r}\subset Y$ \prettyref{eq:Yr def}
to be a union of embedded submanifolds along which the curvature vanishes
non-degenerately \prettyref{eq:non-degeneracy assumption}, the counting
function satisfies the asymptotics 
\begin{equation}
N\left[c_{1}k^{2/r},c_{2}k^{2/r}\right]\sim k^{\frac{d_{j}^{\textrm{max}}}{r}}\sum_{d_{j}=d_{j}^{\textrm{max}}}\int_{NY_{r,j}}\chi_{\left[c_{1},c_{2}\right]}\left(\Delta_{y}\right)\left(v,v\right),\quad\textrm{as }k\rightarrow\infty.\label{eq:Weyl law}
\end{equation}
If further $Y_{r}$ is a finite set of points (or $d_{j}^{\textrm{max}}=0$),
then the smallest eigenvalue of the Bochner Laplacian has a complete
asymptotic expansion\footnote{The same result holds for the $m$th eigenvalue $\lambda_{m}\left(k\right)$
for any fixed $m\in\mathbb{N}_{0}$.} 
\begin{equation}
\lambda_{0}\left(k\right)=k^{2/r}\left[\sum_{l=0}^{N}\lambda_{l}k^{-l/r}+O\left(k^{-\left(2N+1\right)/r}\right)\right],\quad\forall N\in\mathbb{N}_{0},\quad\textrm{as }k\rightarrow\infty.\label{eq:expansion for small eigenvalues}
\end{equation}
\end{thm}

Next, we consider the case when $\left(Y,h^{TY}\right)$ is a complex
Hermitian manifold. The line and vector bundles $\left(L,h^{L}\right)$,
$\left(F,h^{F}\right)$ are then assumed to be holomorphic. Taking
$\nabla^{L},\nabla^{F}$ to be the Chern connections, one also has
the associated Kodaira Laplacian 
\[
\Box_{k}^{q}:\Omega^{0,q}\left(Y;F\otimes L^{k}\right)\rightarrow\Omega^{0,q}\left(Y;F\otimes L^{k}\right),\quad0\leq q\leq m,
\]
acting on tensor powers\footnote{Twisting by an additional bundle $F$ is fairly standard in complex
geometry, for instance one is often required to choose $F$ to be
the canonical bundle (see Proposition \prettyref{prop:Dk spectral estimate}
below).}. The first eigenvalue of the above is typically $0$ with $\ker\Box_{k}^{q}=H^{q}\left(X;F\otimes L^{k}\right)$
being cohomological and corresponding to holomorphic sections. The
Bergman kernel $\Pi_{k}^{q}\left(y,y'\right)$ is the Schwartz kernel
of the orthogonal projector $\Pi_{k}^{q}:\Omega^{0,q}\left(Y;F\otimes L^{k}\right)\rightarrow\ker\Box_{k}^{q}$.
Its value on the diagonal is given
\begin{equation}
\Pi_{k}^{q}\left(y,y\right)=\sum_{j=1}^{N_{k}^{q}}\left|s_{j}\left(y\right)\right|^{2},\quad N_{k}^{q}\coloneqq\dim H^{q}\left(X;F\otimes L^{k}\right),\label{eq:Bergman kernel def.}
\end{equation}
in terms of an orthonormal basis $\left\{ s_{j}\right\} _{j=1}^{N_{k}^{q}}$
of $H^{q}\left(X;F\otimes L^{k}\right)$, and thus controls pointwise
norms of sections in $\ker\Box_{k}^{q}$ in the spirit of \prettyref{eq: estimate first eigenfunction}.
To obtain the asymptotics of $\Pi_{k}^{q}\left(y,y\right)$, we specialize
to the case of Riemann surface ($n-1=2$). Furthermore, in addition
to vanishing at finite order \prettyref{eq: curv vanishes finite order},
the curvature is assumed to be semipositive: $R^{L}\left(w,\bar{w}\right)\geq0$,
for all $w\in T^{1,0}Y$. Under these assumptions one has $H^{1}\left(X;F\otimes L^{k}\right)=0$
for $k$ sufficiently large, with the asymptotics of the Bergman kernel
$\Pi_{k}\coloneqq\Pi_{k}^{0}$ being given by the following. 
\begin{thm}
\label{thm:Bergman kernel expansion} Let $Y$ be a compact Riemann
surface and $(L,h^{L})\to Y$ a semipositive line bundle whose curvature
$R^{L}$ vanishes to finite order at any point. Let $(F,h^{F})\to Y$
be another Hermitian holomorphic vector bundle. Then the Bergman kernel
$\Pi_{k}\coloneqq\Pi_{k}^{0}$ has the pointwise asymptotic expansion
on diagonal 
\begin{equation}
\Pi_{k}\left(y,y\right)=k^{2/r_{y}}\left[\sum_{j=0}^{N}c_{j}\left(y\right)k^{-2j/r_{y}}\right]+O\left(k^{-2N/r_{y}}\right),\quad\forall N\in\mathbb{N}.\label{eq:Bergmankernelexpansion}
\end{equation}
Here $c_{j}$ are sections of $\textrm{End}\left(F\right)$, with
the leading term $c_{0}\left(y\right)>0$ being given in terms of
the Bergman kernel of the model Kodaira Laplacian on the tangent space
at $y$ \prettyref{eq: model Kodaira Laplace}. 
\end{thm}

Note that at points where $R^{L}$ is positive one has $r_{y}=2$
and the above expansion recovers the usual Bergman kernel expansion
at these points. The presence of fractional exponents, at points where
the curvature vanishes, given in terms of the order of vanishing,
represents a new feature. It would be desirable to have a more explicit
formula for the leading term $c_{0}$ at vanishing points for the
curvature. The final example \prettyref{exa:(Branched-coverings)}
computes the leading term explicitly in the case of semipositive line
bundles obtained from branched coverings. Finally, we note that unlike
\prettyref{eq: estimate first eigenfunction} the Bergman kernel expansion
\prettyref{eq:Bergmankernelexpansion} does not exhibit any concentration
phenomenon.

\subsection{Background and commentary}

The result of \prettyref{thm:First eigenvalue Bochner} was shown
by Montgomery \cite{Montgomery-95paper} in the case when $Y$ is
a Riemann surface and $R^{L}$ vanishes to first order ($r=3$) along
a curve. The case of non-vanishing curvature ($r=2$), and a special
case of the expansion \prettyref{eq:expansion for small eigenvalues}
for $r\geq2$, can be found in the work of Helffer-Mohamed \cite{Helffer-Mohamed96}.
The problem has since been explored in several further cases. All
such previous works however are more restrictive in dimension, the
curvature $R^{L}$ or the geometry of the manifold and bundles. Our
\prettyref{thm:First eigenvalue Bochner} is the most general leading
asymptotic for the first Bochner eigenvalue. The only assumption is
the finite order of vanishing of the curvature $R^{L}$ and corresponds
to Hörmander's condition on the unit circle of $L$.

The proof here uses the relation of the Bochner Laplacian with the
sub-Riemannian (sR) Laplacian on the unit circle bundle of $L$, this
is a manifestation of the semiclassical/microlocal correspondence
in this context. Asymptotic bounds on the smallest eigenvalue can
be obtained by replacing Guillemin-Uribe's use of the Melin inequality
on the unit circle \cite{Guillemin-Uribe88} by the subelliptic estimate
of Rothschild-Stein \cite{Rothschild-Stein76}. The leading asymptotic
\prettyref{eq: estimate first eigenvalue} however requires understanding
the sharp constant in the subelliptic estimate. Here we instead exploit
a pointwise heat kernel expansion for the sR Laplacian \cite{BenArous-hypheatexpansion,Leandre-hypoelliptic-heat92}
on the circle bundle, this is also consistent with our method for
the other announcements.

The first part of \prettyref{thm: Weyl law =00003D000026 expansion}
is similarly the semiclassical analog of Weyl's law for hypoelliptic
operators of Hörmander-type. The main difficulty here is the non-uniform
nature of the relevant heat kernel expansion which does not immediately
yield heat trace asymptotics. Prior results on hypoelliptic Weyl law's
include the one by Métivier \cite{Metivier-hypspectralfunction} and
the eigenvalue estimates of Fefferman-Phong \cite{Fefferman-Phong83}.

The leading asymptotic of the Bergman kernel for positive line bundles
on a compact complex manifold was first proved in \cite{Tian90} and
later improved to a full expansion in 
\cite{Catlin97-Bergmankernel,Zelditch98-Bergmankernel}
as a consequence of the Boutet de Monvel-Sjöstrand parametrix 
\cite{Boutet-Sjostrand76}.
Our geometric method here is closer to 
\cite{Dai-Liu-Ma2006-Bergman-kernel,MaMarinescu2008}
and we refer to \cite{Ma-Marinescu} for a detailed account of this
technique and its applications. The problem of the expansion for semipositive
line bundles is well-known and largely unresolved, 
see \cite[Problem 4.8]{Ma-Marinescu} or \cite{Dangelo78}
for the analogous problem for weakly pseudoconvex
domains. Our final \prettyref{thm:Bergman kernel expansion} is the
first instance where the expansion has been proved at vanishing points
of the curvature for surfaces, and this is yet unresolved in higher
dimensions. A key step in our proof of \prettyref{thm:Bergman kernel expansion},
although one among several, is Corollary \prettyref{cor: spectral gap Dirac}
below. This gives an $O\left(k^{2/r}\right)$ spectral gap for the
Kodaira Laplacian on tensor powers by combining \prettyref{thm:First eigenvalue Bochner}
with the method of \cite{MaMarinescu2002}. Donnelly has earlier shown
in \cite{Donnelly2003} that the corresponding does not hold in higher
dimensions as a counterexample to Siu's eigenvalue conjecture \cite{Siu84}
(see Remark \prettyref{rem:good remark} below). Despite the counterexample,
the problem of generalizing \prettyref{thm:Bergman kernel expansion}
to higher dimensions remains open, perhaps by circumventing the use
of Corollary \prettyref{cor: spectral gap Dirac}. Previously, \cite{Berndtsson2002}
proved an asymptotic estimate for the Bergman kernel of semipositive
line bundles. In \cite{Berman-2009} the expansion is proved on the
positive part, and away from the augmented base locus, assuming the
line bundle to be ample. In \cite{Hsiao-Marinescu-2014} the expansion
is proved on the positive part when one twists by the canonical bundle
(i.e.\ $F=K_{Y}$). The analogous problem of the boundary expansion
for the Bergman kernel of weakly pseudoconvex domains in $\mathbb{C}^{2}$
has also been recently solved by the second author in \cite{HsiaoSavale-2022},
refining earlier estimates on Bergman kernels from 
\cite{McNeal89,Nagel-Rosay-Stein-Wainger-89}.

The analysis of holomorphic sections and the Bergman kernel for positive
line bundles has several applications, particularly to the Tian-Yau-Donaldson
program in Kähler geometry, Berezin-Toeplitz quantization, holomorphic
torsion and its relation to Arakelov geometry, random holomorphic
sections and the quantum Hall effect (see \cite{Ma-Marinescu} for
these references). Our \prettyref{thm:Bergman kernel expansion} opens
the way to extending these applications of Bergman kernels to the
case of semipositive line bundles, which we plan to explore in a sequel
to this article. 

\subsection{Organization of the article}

The paper is organized as follows. In \prettyref{sec:sub-Riemannian-geometry}
we begin with some standard preliminaries on sub-Riemannian geometry
and the sR Laplacian. In particular \prettyref{subsec:sR-heat-kernel}
gives a proof of the on-diagonal expansion for the sR heat kernel.
In \prettyref{sec:Bochner-Laplacian-on} we specialize to the case
of sR structures on unit circle bundles. Here Section \prettyref{subsec:Smallest-eigenvalue}
proves \prettyref{thm:First eigenvalue Bochner} based on an analogous
heat kernel expansion for the Bochner Laplacian on tensor powers \prettyref{thm:Bochner heat kernel expansion}.
Next Section \prettyref{subsec:Weyl-laws} and Section \prettyref{subsec:First-eigenvalue expansion}
prove the Weyl law and expansion of the first eigenvalue of \prettyref{thm: Weyl law =00003D000026 expansion}
respectively. In Section \prettyref{subsec:Kodaira-Laplacian} we
come to the case of the Kodaira Laplacian on tensors powers of semipositive
line bundles on a Riemann surface. Here we prove the Bergman kernel
expansion \prettyref{thm:Bergman kernel expansion} in Section \prettyref{subsec:Bergman-kernel exp.}. 

\section{\label{sec:sub-Riemannian-geometry}sub-Riemannian geometry}

Sub-Riemannian (sR) geometry is the study of (metric-)distributions
in smooth manifolds. More precisely, let $X^{n}$ be an $n$-dimensional,
compact, oriented differentiable manifold $X$. Let $E^{m}\subset X$
be a rank $m$ subbundle of the tangent bundle which is assumed to
be\textit{ bracket generating:} sections of $E$ generate all sections
of $TX$ under the Lie bracket. The subbundle $E$ is further equipped
with a metric $g^{E}$. We refer to the triple $\left(X,E,g^{E}\right)$
as a sub-Riemannian (sR) structure. Riemannian geometry corresponds
to $E=TX$.

The obvious length function 
$l\left(\gamma\right)\coloneqq\int_{0}^{1}\left|\dot{\gamma}\right|dt$
may be defined on the set of \textit{horizontal} paths of Sobolev
regularity one connecting the two points $x_{0},x_{1}\in X$ as 
\[
\Omega_{E}\left(x_{0},x_{1}\right)
\coloneqq\left\{ \gamma\in H^{1}\left(\left[0,1\right];X\right)|
\gamma\left(0\right)
=x_{0},\,\gamma\left(1\right)=x_{1},
\,\dot{\gamma}\left(t\right)\in E_{\gamma\left(t\right)}\textrm{ a.e.}\right\} .
\]
This also defines the sub-Riemannian distance function via 
\begin{equation}
d^{E}\left(x_{0},x_{1}\right)\coloneqq\inf_{\gamma\in\Omega_{E}\left(x_{0},x_{1}\right)}l\left(\gamma\right).\label{eq:sR distance}
\end{equation}
The Chow--Rashevskii theorem shows that this distance is finite,
or that there exists a horizontal path connecting any two points on
$X$, giving the manifold the structure of a metric space $\left(X,d^{E}\right)$.

The canonical flag of the distribution $E$ at any point $x\in X$
is defined by
\begin{equation}
\underbrace{E_{0}\left(x\right)}_{=\left\{ 0\right\} }\subset\underbrace{E_{1}\left(x\right)}_{=E}\subset\ldots\subset\subsetneq E_{r\left(x\right)}\left(x\right)=TX\label{eq: canonical flag}
\end{equation}
where $E_{j+1}\coloneqq E_{j}+\left[E_{j},E_{j}\right],\;0\leq j\leq r\left(x\right)-1$
denotes the span of the $j$th brackets. The number $r\left(x\right)$
is called the step or degree of nonholonomy of the distribution at
$x$ and in general depends on the point $x\in X$. Furthermore, the
ranks of the subspaces $E_{j}\left(x\right)$ might also might depend
on $x\in X$ and $E_{j}$ need not define a locally trivial vector
bundles. The growth vector, weight vectors and Hausdorff dimension
of the distribution at $x\in X$ are defined via
\begin{align}
m^{E}\left(x\right) & =\Big(\underbrace{m_{0}^{E}}_{\coloneqq0},\underbrace{m_{1}^{E}}_{=m},m_{2}^{E},\ldots,\underbrace{m_{r}^{E}}_{=n}\Big),\quad\textrm{with }\:m_{j}^{E}\left(x\right)\coloneqq\textrm{dim }E_{j}\left(x\right),\label{eq:growth vector}\\
w^{E}\left(x\right) & =\left(w_{1}^{E}\left(x\right),\ldots,w_{n}^{E}\left(x\right)\right)\coloneqq\Big(\underbrace{1,\ldots,1}_{m_{1}^{E}\textrm{ times}},\underbrace{2,\ldots2}_{m_{2}^{E}-m_{1}^{E}\textrm{ times}},\ldots,\underbrace{r,\ldots,r}_{m_{r}^{E}-m_{r-1}^{E}\textrm{ times}}\Big)\label{eq:weight vector}\\
Q\left(x\right) & \coloneqq\sum_{j=1}^{m}j\left(m_{j}^{E}\left(x\right)-m_{j-1}^{E}\left(x\right)\right)=\sum_{j=1}^{n}w_{j}^{E}\left(x\right).\label{eq:Hausdorff dimension}
\end{align}
A point $x\in X$ is called regular if $m_{j}^{E}$'s are locally
constant functions near $x$ or each distribution $E_{j}$ is a locally
trivial vector bundle near $x$. Mitchell's measure theorem shows
that $Q\left(x\right)$ agrees with the Hausdorff dimension of $\left(X,d^{E}\right)$
as a metric space at a regular point $x\in X$. We call the distribution
$E$ \textit{equiregular} if each point $x\in X$ is regular. Hence
in the equiregular case each $E_{j}$ is a subbundle of $TX$ with
$r\left(x\right)$, $m_{j}^{E}\left(x\right)$ and $Q\left(x\right)$
all being constants independent of $x$.

An important notion is that of a privileged coordinate system at $x$.
To define this, fix a set of local orthonormal frame of vector fields
$U_{1},U_{2},\ldots U_{m}$ for $E$ near $x$. The $E-$order $\textrm{ord}_{E,x}\left(f\right)$
of a function $f\in C^{\infty}\left(X\right)$ at a point $x\in X$
is the maximum integer $s\in\mathbb{N}_{0}$ for which $\sum_{j=1}^{m}s_{j}=s$
implies that $\left(U_{1}^{s_{1}}\ldots U_{m}^{s_{m}}f\right)\left(x\right)=0$.
Similarly the $E-$order $\textrm{ord}_{E,x}\left(P\right)$ of a
differential operator $P$ at the point $x\in X$ is the maximum integer
for which $\textrm{ord}_{E,x}\left(Pf\right)\geq\textrm{ord}_{E,x}\left(P\right)+\textrm{ord}_{E,x}\left(f\right)$
holds for each function $f\in C^{\infty}\left(X\right)$. One then
has the obvious relation $\textrm{ord}_{E,x}\left(PQ\right)\geq\textrm{ord}_{E,x}\left(P\right)+\textrm{ord}_{E,x}\left(Q\right)$
for any pair of differential operators $P,Q$. A set of coordinates
$\left(x_{1},\ldots,x_{n}\right)$ near a point $x\in X$ is said
to be \textit{privileged} if each $x_{j}$ has $E$-order $w_{j}^{E}\left(x\right)$
at $x$. A privileged coordinate system always exists near any point
\cite[pg. 36]{Bellaiche-book-1996}. Furthermore, the coordinate system
may be chosen such that each $\frac{\partial}{\partial x_{j}}$ equals
the value of some bracket monomial in the generating vector fields
at $x$. The $E-$order of the monomial $x^{\alpha}$ in privileged
coordinates is clearly $w.\alpha$, while the defining vector fields
$U_{j}$ all have $E-$order $-1$. A basic vector field is one of
the form $x^{\alpha}\partial_{x_{j}}$ for some $j$ and has $E-$order
$w.\alpha-w_{j}$. We may then use a Taylor expansion to write $U_{j}=\sum_{q=-1}^{\infty}\hat{U}_{j}^{\left(q\right)}$
with each vector field $\hat{U}_{j}^{\left(q\right)}$ being a sum
of basic vector fields of $E$-order $q$. If one defines the rescaling/dilation
$\delta_{\varepsilon}x=\left(\varepsilon^{w_{1}}x_{1},\ldots,\varepsilon^{w_{n}}x_{n}\right)$
in privileged coordinates, the vector fields $\hat{U}_{j}^{\left(q\right)}$
are those appearing in the corresponding expansion $\left(\delta_{\varepsilon}\right)_{*}U_{j}=\sum_{q=-1}^{\infty}\varepsilon^{q}\hat{U}_{j}^{\left(q\right)}$
for the defining vector fields. A differential operator $P$ on $\mathbb{R}^{n}$
is said to be $E-$homogeneous of $\textrm{ord}_{E}\left(P\right)$
iff $\left(\delta_{\varepsilon}\right)_{*}P=\varepsilon^{\textrm{ord}_{E}\left(P\right)}P$.
It is clear that the product of two such homogeneous differential
operators $P_{1},P_{2}$ is homogeneous of $\textrm{ord}_{E}\left(P_{1}P_{2}\right)=\textrm{ord}_{E}\left(P_{1}\right)+\textrm{ord}_{E}\left(P_{2}\right)$.
The \textit{nilpotentization} of the sR structure at an arbitrary
$x\in X$ is the sR manifold given via $\hat{X}=\mathbb{R}^{n}$,$\hat{E}\coloneqq\mathbb{R}\left[\hat{U}_{1}^{\left(-1\right)},\ldots,\hat{U}_{m}^{\left(-1\right)}\right]$
with the metric $\hat{g}^{E}$ corresponding to the identification
$\hat{U}_{j}^{\left(-1\right)}\mapsto\left(U_{j}\right)_{x}$. The
nilpotentization $\hat{\mu}$ of a smooth measure $\mu$ at $x$ is
also defined as the leading part $\hat{\mu}=\hat{\mu}^{\left(0\right)}$
under the privileged coordinate expansion $\left(\delta_{\varepsilon}\right)_{*}\mu=\varepsilon^{Q\left(x\right)}\left[\sum_{q=0}^{\infty}\hat{\mu}^{\left(q\right)}\right]$.
These nilpotentizations can be shown to be independent of the choice
of privileged coordinates up to sR isometry \cite[Ch. 5]{Bellaiche-book-1996}. 

\subsection{\label{subsec:sR-Laplacian}sR Laplacian}

Here we define the sub-Riemannian (sR) Laplacian and state its basic
properties. It shall be useful to define it as acting on sections
of an auxiliary complex Hermitian vector bundle of rank $l$ with
connection $\left(F,h^{F},\nabla^{F}\right)$ . To define this first
define the sR-gradient $\nabla^{g^{E},F}s\in C^{\infty}\left(X;E\otimes F\right)$
of a section $s\in C^{\infty}\left(X;F\right)$ by the equation 
\begin{equation}
h^{E,F}\left(\nabla^{g^{E},F}s,v\otimes s'\right)\coloneqq h^{F}\left(\nabla_{v}^{F}s,s'\right),\quad\forall v\in C^{\infty}\left(X;E\right),s'\in C^{\infty}\left(X;F\right),\label{eq:gradient}
\end{equation}
where $h^{E,F}\coloneqq g^{E}\otimes h^{F}$. Next, one defines the
divergence or adjoint of this gradient. In the sR context, the lack
of canonical volume form presents a difficulty in doing so, hence
we shall choose an auxiliary non-vanishing volume form $\mu$. The
divergence $\left(\nabla^{g^{E},F}\right)_{\mu}^{*}\omega\in C^{\infty}\left(X;F\right)$
of a section $\omega\in C^{\infty}\left(X;E\otimes F\right)$ is now
defined to be the formal adjoint satisfying 
\begin{equation}
\int\left\langle \left(\nabla^{g^{E},F}\right)_{\mu}^{*}\omega,s\right\rangle \mu=-\int\left\langle \omega,\nabla^{g^{E},F}s\right\rangle \mu,\quad\forall s\in C^{\infty}\left(X;F\right).\label{eq:divergence}
\end{equation}
The \textbf{sR-Laplacian }acting on sections of $F$ is then defined
by the composition 
\begin{equation}
\Delta_{g^{E},F,\mu}\coloneqq\left(\nabla^{g^{E},F}\right)_{\mu}^{*}\circ\nabla^{g^{E},F}:C^{\infty}\left(X;F\right)\rightarrow C^{\infty}\left(X;F\right).\label{eq:sR Laplacian}
\end{equation}
In terms of a local orthonormal frame $\left\{ U_{j}\right\} _{j=1}^{m}$
for $E$, we have the expression 
\begin{equation}
\Delta_{g^{E},F,\mu}s=\sum_{j=1}^{m}\left[-\left(\nabla_{U_{j}}^{F}\right)^{2}s+\left(\nabla^{g^{E}}U_{j}\right)_{\mu}^{*}\nabla_{U_{j}}^{F}s\right]\label{eq:local expression for Delta}
\end{equation}
with $\left(\nabla^{g^{E}}U_{j}\right)_{\mu}^{*}$ being the divergence
of the vector field $U_{j}$ with respect to $\mu$. Changing the
volume form $\mu$ changes the sR Laplacian \prettyref{eq:sR Laplacian}
by a conjugate, up to a term of order zero.

The sR Laplacian $\Delta_{g^{E},F,\mu}$ is non-negative and self-adjoint
with respect to the obvious inner product $\left\langle s,s'\right\rangle =\int_{X}h^{F}\left(s,s'\right)\mu$,
$s,s'\in C^{\infty}\left(X;F\right)$. Its principal symbol is easily
computed to be the Hamiltonian 
\begin{equation}
\sigma=\sigma\left(\Delta_{g^{E},F,\mu}\right)\left(x,\xi\right)=H^{E}\left(x,\xi\right)=\left|\left.\xi\right|_{E}\right|^{2}\in C^{\infty}\left(T^{*}X\right)\label{eq:symbol sr Laplace-1}
\end{equation}
while its sub-principal symbol is zero. The characteristic variety
\begin{align}
\Sigma_{\Delta_{g^{E},F,\mu}} & =\left\{ \left(x,\xi\right)\in T^{*}X|\sigma\left(\Delta_{g^{E},F,\mu}\right)\left(x,\xi\right)=0\right\} =\left\{ \left(x,\xi\right)|\left.\xi\right|_{E}=0\right\} \eqqcolon E^{\perp}\label{eq:characteristic variety}
\end{align}
is the annihilator of $E$. From the local expression \prettyref{eq:local expression for Delta}
and the bracket generating condition on $E$, the Laplacian $\Delta_{g^{E},F,\mu}$
is seen to be a sum of squares operator of Hörmander type \cite{Hormander1967}.
It is then known to be hypoelliptic and satisfies the following optimal
sub-elliptic estimate \cite{Rothschild-Stein76} with a gain of $\frac{1}{r}$
derivatives 
\begin{equation}
\left\Vert \psi s\right\Vert _{H^{1/r}}^{2}\leq C\left[\left\langle \Delta_{g^{E},F,\mu}\varphi s,\varphi s\right\rangle +\left\Vert \varphi s\right\Vert _{L^{2}}^{2}\right],\quad\forall s\in C^{\infty}\left(X;F\right)\label{eq:local subelliptic estimate}
\end{equation}
for all $\varphi,\psi\in C_{c}^{\infty}\left(X\right)$, with $\varphi=1$
on $\textrm{spt}\left(\psi\right)$, and where $r\coloneqq\sup_{x\in X}\,r\left(x\right)$
is the maximum step size of the distribution. 

Thus on a compact manifold the sR Laplacian has a compact resolvent,
a discrete spectrum of non-negative eigenvalues approaching infinity
and a well-defined heat operator $e^{-t\Delta_{g^{E},F,\mu}}$. 

\subsubsection{\label{subsec:sR-heat-kernel} sR heat kernel}

We shall now describe the asymptotics of the heat kernel $e^{-t\Delta_{g^{E},F,\mu}}.$
One first begins with the finite propagation speed for the sR wave
equation \cite{Melrose-subelliptic-propagation}: the Schwartz kernel
$K_{t}\left(x,y\right)$ of the half-wave operator $e^{it\sqrt{\Delta_{g^{E},F,\mu}}}$
is supported
\begin{equation}
\textrm{spt }K_{t}\subset\left\{ \left(x,y\right)|d^{E}\left(x,y\right)\leq\left|t\right|\right\} \label{eq:finite propagation speed}
\end{equation}
in a $\left|t\right|$ size neighborhood of the diagonal measured
with respect to the sR distance \prettyref{eq:sR distance}. From
this one obtains a localization for the heat kernel. To state it,
fix a Riemannian metric $g^{TX}$ and a privileged coordinate ball
$B_{\varrho}^{g^{TX}}\left(x\right)$, centered at a point $x$, of
radius $\varrho_{x}$ depending on $x$. Let $U_{1},\ldots,U_{m}$
be a local orthonormal frame for $E$ on this ball. Let $\chi\in C_{c}^{\infty}\left(\left[-1,1\right];\left[0,1\right]\right)$
with $\chi=1$ on $\left[-\frac{1}{2},\frac{1}{2}\right]$. Define
the modified measure and vector fields on $\mathbb{R}^{n}$ via
\begin{align}
\tilde{\mu} & =\hat{\mu}+\chi\left(\frac{d^{g^{TX}}\left(x,x'\right)}{\varrho_{x}}\right)\left(\mu-\hat{\mu}\right),\nonumber \\
\tilde{U}_{j} & =U_{j}^{\left(-1\right)}+\chi\left(\frac{d^{g^{TX}}\left(x,x'\right)}{\varrho_{x}}\right)\left(U_{j}-U_{j}^{\left(-1\right)}\right),\quad1\leq j\leq m,\label{eq:modified vector fields}
\end{align}
in terms of the nilpotentization at $x$ given by these privileged
coordinates. These modified vector fields can be seen to be bracket
generating for $\varrho$ sufficiently small. The connection on $F$
can be written $\nabla^{F}=d+A$, in terms of an orthonormal trivialization
for $F$ over the ball, where $A\in\Omega^{1}\left(B_{\varrho}^{g^{TX}}\left(x\right);\mathfrak{u}\left(l\right)\right)$,
$A\left(0\right)=0$, is a one form valued in the Lie algebra $\mathfrak{u}\left(l\right)$
of the unitary group. A modified sR metric $\tilde{g}^{E}$ on $\mathbb{R}^{n}$
is now defined by requiring the vector fields \prettyref{eq:modified vector fields}
to be orthonormal. While a modified connection on the trivial vector
bundle of rank $\textrm{rk}\left(F\right)$ on $\mathbb{R}^{n}$ is
defined as $\tilde{\nabla}^{F}\coloneqq d+\chi\left(\frac{d^{E}\left(x,x'\right)}{\varrho}\right)A$.
A formula similar to \prettyref{eq:local expression for Delta} now
gives an sR Laplacian on $\mathbb{R}^{n}$ via 
\[
\tilde{\Delta}_{g^{E},F,\mu}s=\sum_{j=1}^{m}\left[-\left(\tilde{\nabla}_{\tilde{U}_{j}}^{F}\right)^{2}s+\left(\nabla^{\tilde{g}^{E}}\tilde{U}_{j}\right)_{\tilde{\mu}}^{*}\tilde{\nabla}_{\tilde{U}_{j}}^{F}s\right].
\]
Being semi-bounded from below, it is essentially self-adjoint and
has a well-defined heat kernel on $\mathbb{R}^{n}$ using functional
calculus. Furthermore from the bracket generation of \prettyref{eq:modified vector fields},
it is of Hörmander type and satisfies a local sub-elliptic estimate
\prettyref{eq:local subelliptic estimate}. Next, an application of
finite propagation speed for the wave operator \prettyref{eq:finite propagation speed}
gives localization for the heat kernel for the sR Laplacian. Namely,
there exist constants $\rho_{1,x},C_{x}$ depending on $x$ such that
\begin{align}
e^{-t\Delta_{g^{E},F,\mu}}\left(x,x'\right) & \leq Ct^{-2nr-1}e^{-\frac{d^{E}\left(x,x'\right)^{2}}{4t}}\label{eq:heat exponential decay off-diagonal}\\
e^{-t\Delta_{g^{E},F,\mu}}\left(x,x'\right)-e^{-t\tilde{\Delta}_{g^{E},F,\mu}}\left(x,x'\right) & \leq C_{x}e^{-\frac{\varrho_{1,x}^{2}}{16t}}\label{eq: heat localization on diagonal}
\end{align}
 for $d^{E}\left(x,x'\right)\leq\varrho_{1,x}$ and $t\leq1$. 

We now have the following on diagonal expansion for the sR heat kernel. 
\begin{thm}
\label{thm: BenArousThm} There exist smooth sections $A_{j}\in C^{\infty}\left(X;\textrm{End}\left(F\right)\right)$such
that 
\begin{equation}
\left[e^{-t\Delta_{g^{E},F,\mu}}\right]_{\mu}\left(x,x\right)=\frac{1}{t^{Q\left(x\right)/2}}\left[A_{0}\left(x\right)+A_{1}\left(x\right)t+\ldots+A_{N}\left(x\right)t^{N}+O\left(t^{N}\right)\right]\label{eq:Ben Arous exp.}
\end{equation}
$\forall x\in X$, $N\in\mathbb{N}$. The leading term $A_{0}=\left[e^{-\hat{\Delta}_{\hat{g}^{E},\hat{\mu}}}\right]_{\hat{\mu}}\left(0,0\right)$
is a multiple of the identity and given in terms of the scalar heat
kernel on the nilpotent approximation at $x$. 
\end{thm}

\begin{proof}
By \prettyref{eq: heat localization on diagonal} it suffices to demonstrate
the expansion for the localized heat kernel $e^{-t\tilde{\Delta}_{g^{E},F,\mu}}\left(0,0\right)$
at the point $x$. Next, the heat kernel of the rescaled sR-Laplacian
\begin{equation}
\tilde{\Delta}_{g^{E},F,\mu}^{\varepsilon}\coloneqq\varepsilon^{2}\left(\delta_{\varepsilon}\right)_{*}\tilde{\Delta}_{g^{E},F,\mu}\label{eq:rescaled sR Laplacian}
\end{equation}
under the privileged coordinate dilation satisfies 
\begin{equation}
e^{-t\tilde{\Delta}_{g^{E},F,\mu}^{\varepsilon}}\left(x,x'\right)=\varepsilon^{Q\left(x\right)}e^{-t\varepsilon^{2}\tilde{\Delta}_{g^{E},F,\mu}}\left(\delta_{\varepsilon}x,\delta_{\varepsilon}x'\right).\label{eq:coordinate dilation}
\end{equation}
Rearranging and setting $x=x'=0,\,t=1$; gives 
\[
\varepsilon^{-Q\left(x\right)}e^{-\tilde{\Delta}_{g^{E},F,\mu}^{\varepsilon}}\left(0,0\right)=e^{-\varepsilon^{2}\tilde{\Delta}_{g^{E},F,\mu}}\left(0,0\right)
\]
and it suffices to compute the expansion of the left-hand side above
as the dilation $\varepsilon\rightarrow0$. To this end, first note
that the rescaled Laplacian has an expansion under the privileged
coordinate dilation 
\begin{equation}
\tilde{\Delta}_{g^{E},F,\mu}^{\varepsilon}=\left(\sum_{j=0}^{N}\varepsilon^{j}\hat{\Delta}_{g^{E},F,\mu}^{\left(j\right)}\right)+\varepsilon^{N+1}R_{\varepsilon}^{\left(N\right)},\quad\forall N.\label{eq: privileged cord expansion Delta}
\end{equation}
Here each $\hat{\Delta}_{g^{E},F,\mu}^{\left(j\right)}$ is an $\varepsilon$-independent
second-order differential operator on $\mathbb{R}^{n}$ of homogeneous
$E-$order $j-2$. While each $R_{\varepsilon}^{\left(N\right)}$
is an $\varepsilon$-dependent second-order differential operator
on $\mathbb{R}^{n}$ of $E$-order at least $N-1$. The coefficient
functions of $\hat{\Delta}_{g^{E},F,\mu}^{\left(j\right)}$ are polynomials
of degree at most $j+2r$. While those of $R_{\varepsilon}^{\left(N\right)}$
are uniformly $C^{\infty}$-bounded in $\varepsilon$. The first term
is a scalar operator given in terms of the nilpotent approximation
at $x$
\begin{equation}
\hat{\Delta}_{g^{E},F,\mu}^{\left(0\right)}=\Delta_{\hat{g}^{E},\hat{\mu};x}=\sum_{j=1}^{m}\left(\hat{U}_{j}^{\left(-1\right)}\right)^{2}.\label{eq: nilpotent Laplacian Euclidean}
\end{equation}
This expansion \prettyref{eq: privileged cord expansion Delta} along
with the subelliptic estimate \prettyref{eq:local subelliptic estimate}
now gives 
\[
\left(\tilde{\Delta}_{g^{E},F,\mu}^{\varepsilon}-z\right)^{-1}-\left(\hat{\Delta}_{g^{E},F,\mu}^{\left(0\right)}-z\right)^{-1}=O_{H_{\textrm{loc}}^{s}\rightarrow H_{\textrm{loc}}^{s+1/r-2}}\left(\varepsilon\left|\textrm{Im}z\right|^{-2}\right),
\]
$\forall s\in\mathbb{R}$. More generally, we let $I_{j}\coloneqq\left\{ p=\left(p_{0},p_{1},\ldots\right)|p_{\alpha}\in\mathbb{N},\sum p_{\alpha}=j\right\} $
denote the set of partitions of the integer $j$ and define 
\begin{equation}
\mathtt{C}_{j}^{z}\coloneqq\sum_{p\in I_{j}}\left(\hat{\Delta}_{g^{E},F,\mu}^{\left(0\right)}-z\right)^{-1}\left[\prod_{\alpha}\hat{\Delta}_{g^{E},F,\mu}^{\left(p_{\alpha}\right)}\left(\hat{\Delta}_{g^{E},F,\mu}^{\left(0\right)}-z\right)^{-1}\right].\label{eq:resolvent exp. terms}
\end{equation}
Then by repeated applications of the subelliptic estimate we have
\begin{equation}
\left(\tilde{\Delta}_{g^{E},F,\mu}^{\varepsilon}-z\right)^{-1}-\sum_{j=0}^{N}\varepsilon^{j}\mathtt{C}_{j}^{z}=O_{H_{\textrm{loc}}^{s}\rightarrow H_{\textrm{loc}}^{s+N\left(1/r-2\right)}}\left(\varepsilon^{N+1}\left|\textrm{Im}z\right|^{-2Nw_{n}-2}\right),\label{eq:sR resolvent expansion}
\end{equation}
$\forall s\in\mathbb{R}$. A similar expansion as \prettyref{eq: privileged cord expansion Delta}
for $\left(\tilde{\Delta}_{g^{E},F,\mu}^{\varepsilon}+1\right)^{M}\left(\tilde{\Delta}_{g^{E},F,\mu}^{\varepsilon}-z\right)$,
$M\in\mathbb{N}$, also gives 
\begin{equation}
\left(\tilde{\Delta}_{g^{E},F,\mu}^{\varepsilon}+1\right)^{-M}\left(\tilde{\Delta}_{g^{E},F,\mu}^{\varepsilon}-z\right)^{-1}-\sum_{j=0}^{N}\varepsilon^{j}\mathtt{C}_{j,M}^{z}=O_{H_{\textrm{loc}}^{s}\rightarrow H_{\textrm{loc}}^{s+N\left(1/r-2\right)+\frac{M}{r}}}\left(\varepsilon^{N+1}\left|\textrm{Im}z\right|^{-2Nw_{n}-2}\right)\label{eq: regularized expansion}
\end{equation}
where $\mathtt{C}_{j,M}^{z}=O_{H_{\textrm{loc}}^{s}\rightarrow H_{\textrm{loc}}^{s+N\left(1/r-2\right)+\frac{M}{r}}}\left(\varepsilon^{N+1}\left|\textrm{Im}z\right|^{-2Nw_{n}-2}\right)$,
$j=0,\ldots,N$, with 
\[
\mathtt{C}_{0,M}^{z}=\left(\hat{\Delta}_{g^{E},F,\mu}^{\left(0\right)}+1\right)^{-M}\left(\hat{\Delta}_{g^{E},F,\mu}^{\left(0\right)}-z\right)^{-1}.
\]
For $M\gg0$ sufficiently large, Sobolev's inequality gives an expansion
for the corresponding Schwartz kernels of \prettyref{eq: regularized expansion}
in $C^{0}\left(\mathbb{R}^{n}\times\mathbb{R}^{n}\right)$. The heat
kernel expansion now follows by plugging the resolvent expansion into
the Helffer-Sjöstrand formula (see \cite[Ch. 8, eq. 8.3]{Dimassi-Sjostrand}
for this formula and the notion of an analytic continuation used therein).
Finally, to see that the expansion only involves even powers of $\varepsilon$,
or that \prettyref{eq:Ben Arous exp.} has no half-integer powers
of $t$, note that the operators $\hat{\Delta}_{g^{E},F,\mu}^{\left(j\right)}$
in the expansion \prettyref{eq: privileged cord expansion Delta}
change sign by $\left(-1\right)^{j}$ under the rescaling $\delta_{-1}$.
Thus the Schwartz kernel for $\mathtt{C}_{j}^{z}$ \prettyref{eq:resolvent exp. terms}
then changes sign by $\left(-1\right)^{j}$ under this change of variables
giving $\mathtt{C}_{j}^{z}\left(0,0\right)=0$ for $j$ odd. 
\end{proof}
The above proof similarly gives an expansion for functions of the
Laplacian 
\begin{equation}
\left[\varphi\left(t\Delta_{g^{E},F,\mu}\right)\right]_{\mu}\left(x,x\right)=\frac{1}{t^{Q\left(x\right)/2}}\left[A_{0}^{\varphi}\left(x\right)+A_{1}^{\varphi}\left(x\right)t+\ldots+A_{N}^{\varphi}\left(x\right)t^{N}+O\left(t^{N}\right)\right],\label{eq:functional expansion.}
\end{equation}
$\forall\varphi\in\mathcal{S}\left(\mathbb{R}\right)$. As usual,
the same proof gives a point-wise, near-off diagonal expansion for
the heat kernel and its derivatives: i.e. an asymptotic expansion
for $\left[\varphi\left(t\Delta_{g^{E},F,\mu}\right)\right]_{\mu}\left(\delta_{t^{\frac{1}{2}}}x,\delta_{t^{\frac{1}{2}}}x'\right)$,
as $t\rightarrow0$, on the chosen privileged coordinate ball in the
$C^{\infty}$-norm on the product. This is only a matter of different
substitutions in \prettyref{eq:coordinate dilation} and in the Helffer-Sjöstrand
formula for $\varphi$ in \prettyref{eq: regularized expansion}. 

However both the above and the expansion \prettyref{thm: BenArousThm}
hold only pointwise along the diagonal. In particular, the leading
order $Q\left(x\right)$ is in general a function of the point $x$
on the diagonal. This hence does not immediately give heat trace or
spectral function asymptotics for the sR Laplacian as the expansion
might not be uniform or integrable in $x$. In the equiregular case,
where $Q\left(x\right)=Q$ is constant, a uniform set of privileged
coordinates, privileged at each point in a neighborhood of $x$, may
be chosen in the proof. This gives the uniformity of the expansion
in $x$ and one obtains the asymptotics for the Weyl counting function
$N\left(\lambda\right)$, for the number of eigenvalues of $\Delta_{g^{E},F,\mu}$
below $\lambda$. 
\begin{thm}
\label{thm:equiregular expansion}For an equiregular sR manifold case
there is a heat trace expansion 
\[
\textrm{tr }e^{-t\Delta_{g^{E},F,\mu}}=\frac{1}{t^{Q/2}}\left[a_{0}+a_{1}t+\ldots+a_{N}t^{N}+O\left(t^{N}\right)\right],
\]
$\forall N\in\mathbb{N},$ with leading term given by $a_{0}=\int_{X}\left[e^{-\hat{\Delta}_{\hat{g}^{E},\hat{\mu}}}\right]_{\hat{\mu}}\left(0,0\right)\mu$.
Thus the Weyl counting function satisfies 
\begin{align*}
N\left(\lambda\right) & =\frac{\lambda^{Q/2}\left(1+o\left(1\right)\right)}{\Gamma\left(Q/2+1\right)}\int_{X}\left[e^{-\hat{\Delta}_{\hat{g}^{E},\hat{\mu}}}\right]_{\hat{\mu}}\left(0,0\right)\mu.
\end{align*}
\end{thm}

The above two theorems are by now well known \cite{BenArous-hypheatexpansion,Leandre-hypoelliptic-heat92,Metivier-hypspectralfunction,Takanobu-hypheat},
with the investigation of the small time heat kernel asymptotics having
begun in \cite{Bismut84}. The above proof is based on the analytic
localization technique \cite{Bismut-Lebeau91} combined with the use
of sR geometric privileged coordinate dilations.

\section{\label{sec:Bochner-Laplacian-on}Bochner Laplacian on tensor powers}

A natural place where sub-Riemannian structures arise is on unit circle
bundles. To be precise, let us consider $\left(X,E,g^{E}\right)$
a corank $1$ sR structure on an $n$-dimensional manifold $X$. We
assume that there is a free $S^{1}$ action on $X$ with respect to
which the sR structure is invariant and transversal: the generator
$e\in C^{\infty}\left(TX\right)$ of the action and $E$ are transversal
at each point. The quotient $Y\coloneqq X/S^{1}$ is then a manifold
with a Riemannian metric $g^{TY}$ induced from $g^{E}$. Equivalently,
the natural projection $\pi:X\rightarrow Y$ is a principal $S^{1}$
bundle with connection given by the horizontal distribution $E$.
Let $L\coloneqq X\times_{\rho}S^{1}\rightarrow Y$ be the Hermitian
line bundle associated to the standard one-dimensional representation
$\rho$ of $S^{1}$ with induced connection $\nabla^{L}$ and curvature
$R^{L}$. Since the distribution is of corank $1$, the growth vector
at $x$ is simply a function of the step $r\left(x\right)$ and given
by $m^{E}\left(x\right)=\Big(0,\underbrace{n-1,n-1,\ldots,n-1}_{r\left(x\right)-1\,\textrm{times}},n\Big)$
\prettyref{eq:growth vector}. Equivalently, the canonical flag \prettyref{eq: canonical flag}
is given by 
\[
E_{j}\left(x\right)=\begin{cases}
E; & 1\leq j\leq r\left(x\right)-1\\
TX; & j=r\left(x\right)
\end{cases}.
\]
 Also, note that the weight vector at $x$ is $w^{E}\left(x\right)=\Big(\underbrace{1,1,\ldots,1}_{n-1\,\textrm{ times}},r\left(x\right)\Big),$
while the Hausdorff dimension is given by $Q\left(x\right)=n-1+r\left(x\right)$.
On account of the $S^{1}$ invariance, each of $m^{E}\left(x\right),$
$r\left(x\right)$ and $Q\left(x\right)$ descend to functions on
the base manifold $Y$. The degree of nonholonomy $r\left(x\right)$
at $x$ is now characterized in terms of the order of vanishing of
the curvature $R^{L}$ as below. 
\begin{prop}
\label{prop: degree of nonholonomy is ord van.}The degree of nonholonomy
of an $S^{1}$ invariant sR structure
\begin{align}
r\left(x\right)-2=\textrm{ord}\left(R^{L}\right) & \coloneqq\min\left\{ l|j_{\pi\left(x\right)}^{l}\left(R^{L}\right)\neq0\right\} \label{eq:order of vanishing}
\end{align}
is given in terms of the order of vanishing of the curvature $R^{L}$
on the base. 
\end{prop}

\begin{proof}
In terms of local coordinates on $Y$ and a local orthonormal section
$\mathtt{l}$ for $L$, we may write $\nabla^{L}=d+ia^{L}$; $a^{L}\in\Omega^{1}\left(Y\right)$,
while $E=\ker\left[d\theta+a^{L}\right]$ with $\theta$ being the
induced coordinate on each fiber of $X$. The proposition now follows
on noting $\left[U_{i},U_{j}\right]=\left(da^{L}\right)_{ij}\partial_{\theta}=R_{ij}^{L}\partial_{\theta}$
for the local generating vector fields $U_{j}\coloneqq\partial_{y_{j}}-a_{j}^{L}\partial_{\theta}$,
$1\leq j\leq n-1$. Repeated brackets among the $U_{j}$'s are then
given in terms of derivatives of the curvature $R^{L}$.
\end{proof}
Thus we see that the bracket generating condition is equivalent to
the curvature $R^{L}$ having a finite order of vanishing at each
point of $Y$. 

\subsubsection{\label{subsec:Structure-of Yr}Structure of $Y_{r}$}

As noted before, the function $y\mapsto r_{y}$ \prettyref{eq:order of vanishing curv.}
is upper semi-continuous and gives a decomposition of the manifold
$Y=\bigcup_{j=2}^{r}Y_{j}$; $Y_{j}\coloneqq\left\{ y\in Y|r_{y}=j\right\} $
with each $Y_{\leq j}\coloneqq\bigcup_{j'=0}^{j}Y_{j'}$ being open.
We next address the local structure of $Y_{r}$, the locus of highest
vanishing order for the curvature. 
\begin{prop}
\label{prop:structure of Yr} The subset $Y_{r}\subset Y$ is locally
any closed subset of a hypersurface. 
\end{prop}

\begin{proof}
First, express the curvature $R^{L}=R_{ij}^{L}dy_{i}\wedge dy_{j}$
in some coordinates centered at $y\in Y_{r}$. By definition, $Y_{r}$
is described by equations of the following type near $y$
\begin{align}
\partial_{y}^{\alpha}R_{ij}^{L} & =0,\quad\forall i,j=1,2,\ldots,n-1,\,\alpha\in\mathbb{N}_{0}^{n-1},\left|\alpha\right|\leq r-3,\quad\textrm{while}\label{eq:local vanishing}\\
\partial_{y}^{\alpha_{0}}R_{i_{0}j_{0}}^{L} & \neq0,\quad\textrm{for some }i_{0},j_{0}=1,2,\ldots,n-1,\,\alpha_{0}\in\mathbb{N}_{0}^{n-1},\left|\alpha_{0}\right|=r-2.\label{eq:local non-vanishing}
\end{align}
The second equation \prettyref{eq:local non-vanishing} implies that
one of the functions $\partial_{y}^{\alpha}R_{ij}^{L}$, $\left|\alpha\right|=r-3$,
has a non-zero differential and cuts out a hypersurface. 

Conversely, let $S\subset\left\{ 0\right\} \times\mathbb{B}_{y_{2},\ldots,y_{n-1}}^{n-2}\subset\mathbb{R}_{y_{1},y_{2},\ldots,y_{n-1}}^{n-1}$
be any closed subset of the $n-2$ dimensional ball, sitting inside
the hypersurface $\left\{ y_{1}=0\right\} $ in $n-1$ dimensions.
By an application of the Whitney extension theorem, there exists a
smooth function $f\left(y_{2},\ldots,y_{n-1}\right)\in C^{\infty}\left(\mathbb{R}_{y_{2},\ldots,y_{n-1}}^{n-2}\right)$
such that $S=\left\{ 0\right\} \times\left\{ f_{y_{3}}=\ldots=f_{y_{n-1}}=0\right\} $.
The closed two form 
\[
R^{L}=d\left[-fdy_{1}+\frac{1}{2}y_{1}^{2}dy_{2}\right]=\left(y_{1}+f_{y_{2}}\right)dy_{1}dy_{2}+\sum_{j=3}^{n-1}f_{y_{j}}dy_{1}dy_{j}
\]
 is the curvature of some connection on the trivial line bundle over
the ball. This curvature form satisfies $r=3$ with $Y_{3}=\left\{ y_{1}+f_{y_{2}}=f_{y_{3}}=\ldots=f_{y_{n-1}}=0\right\} $.
The local structure of $Y_{3}$ near the origin is now the same as
$S$ under the diffeomorphism $\left(y_{1},y_{2}\ldots,y_{n-1}\right)\mapsto\left(y_{1}+f_{y_{2}},y_{2},\ldots,y_{n-1}\right)$.
\end{proof}

\subsection{\label{subsec:Smallest-eigenvalue}Smallest eigenvalue}

The unit circle bundle of $L$ being $X$, the pullback $\mathbb{C}\cong\pi^{*}L\rightarrow X$
is canonically trivial via the identification $\pi^{*}L\ni\left(x,l\right)\mapsto x^{-1}l\in\mathbb{C}$.
Pick an auxiliary complex Hermitian vector bundle with connection
$\left(F,h^{F},\nabla^{F}\right)$ on $Y$ and we denote by the same
notation its pullback to $X$. Pulling back sections then gives the
identification 
\begin{equation}
C^{\infty}\left(X;F\right)=\oplus_{k\in\mathbb{Z}}C^{\infty}\left(Y;F\otimes L^{k}\right).\label{eq:Fourier decomposition}
\end{equation}
Each summand on the right-hand side above corresponds to an eigenspace
of $\nabla_{e}^{F}$ with eigenvalue $-ik$. While horizontal differentiation
$d^{H}$ on the left corresponds to differentiation with respect to
the tensor product connection$\nabla^{L^{k}}$ on the right-hand side
above. Pick an invariant density $\mu_{X}$ on $X$ inducing a density
$\mu_{Y}$ on $Y$. This now defines the sR Laplacian $\Delta_{g^{E},F,\mu_{X}}$
acting on sections of $F$. By invariance the sR Laplacian commutes
$\left[\Delta_{g^{E},F,\mu_{X}},e\right]=0$ with the generator of
the circle action and hence preserves the decomposition \prettyref{eq:Fourier decomposition}.
It acts via 
\begin{equation}
\Delta_{g^{E},F,\mu_{X}}=\oplus_{k\in\mathbb{Z}}\Delta_{k}\label{eq: Fourier decomposition Laplacian}
\end{equation}
on each component where $\Delta_{k}$ is the Bochner Laplacian \eqref{e:BL}
on the tensor powers $F\otimes L^{k}$, with adjoint being taken with
respect to $\mu_{Y}$. 

Next, we show that the heat kernel expansion for the sR Laplacian
\prettyref{thm: BenArousThm} gives a corresponding heat kernel expansion
for the Bochner Laplacian.
\begin{thm}
\label{thm:Bochner heat kernel expansion}The heat kernel of the Bochner
Laplacian $\Delta_{k}$ has the following pointwise expansion on the
diagonal 
\begin{equation}
e^{-\frac{t}{k^{2/r}}\Delta_{k}}\left(y,y\right)=\begin{cases}
k^{\left(n-1\right)/r}\left[\sum_{j=0}^{N}a_{2j}\left(y;t\right)k^{-2j/r}+O\left(k^{-\left(N+1\right)/r}\right)\right]; & y\in Y_{r}\\
O\left(k^{-\infty}\right); & y\in Y_{\leq r-1}
\end{cases}\label{eq: heat kernel expansion-1}
\end{equation}
with leading coefficient $a_{0}\left(y;t\right)=e^{-t\Delta_{y}}\left(0,0\right)$
being the heat kernel of the model operator \prettyref{eq:model Bochner Laplacian}
on the tangent space. 
\end{thm}

\begin{proof}
The Fourier decomposition for the Laplacians \prettyref{eq: Fourier decomposition Laplacian},
gives the corresponding relation 
\begin{equation}
e^{-T\Delta_{k}}\left(y_{1},y_{2}\right)=\left[\int d\theta\,e^{-T\Delta_{g^{E},F,\mu_{X}}}\left(l_{y_{1}},l_{y_{2}}e^{i\theta}\right)e^{-ik\theta}\right]l_{y_{1}}\otimes l_{y_{2}}^{*}\label{eq:Fourier mode heat kernel}
\end{equation}
between the heat kernels with $l_{y_{1}},l_{y_{2}}$ denoting two
unit elements in the fibers of $L$ above $y_{1},y_{2}$ respectively.
We again note that the kernels are computed with respect to the densities
$\mu_{X}$, $\mu_{Y}$ chosen before. The above relation together
with \prettyref{eq:heat exponential decay off-diagonal} first gives
\begin{equation}
e^{-\frac{1}{k^{2/r}}\Delta_{k}}\left(y_{1},y_{2}\right)=c_{\varepsilon,N}k^{-N},\quad\forall N\in\mathbb{N},\label{eq: off diagonal decay Bochner}
\end{equation}
when $d\left(y_{1},y_{2}\right)>\varepsilon>0$.

Choosing a coordinate system centered at a point $y\in Y$ and a local
orthonormal section $\mathtt{l}$ of $L$ gives an induced coordinate
system on the unit circle bundle near $x$. It is easy to see that
this induced coordinate system is privileged at each point on the
fiber above $y.$

Next, using \prettyref{eq:Fourier mode heat kernel} with $T=\varepsilon^{2}t$
and $y_{1}=y_{2}$ belonging to this coordinate chart, one has 
\begin{align*}
e^{-\varepsilon^{2}t\Delta_{k}}\left(y_{1},y_{1}\right) & =\left[\int d\delta_{\varepsilon}\theta'\,e^{-\varepsilon^{2}t\Delta_{g^{E},F,\mu_{X}}}\left(\mathtt{l}\left(y_{1}\right),\mathtt{l}\left(y_{1}\right)e^{i\delta_{\varepsilon}\theta'}\right)e^{-ik\delta_{\varepsilon}\theta'}\right]
\end{align*}
where $\delta_{\varepsilon}$ denotes the privileged coordinate dilation
as before. Now setting $y_{1}=\varepsilon\mathsf{y}=\delta_{\varepsilon}\mathsf{y}$,
the equations \prettyref{eq:coordinate dilation}, \prettyref{eq: privileged cord expansion Delta}
in the proof of \prettyref{thm: BenArousThm} give an expansion for
the integrand above 
\begin{equation}
e^{-\varepsilon^{2}t\Delta_{k}}\left(\delta_{\varepsilon}\mathsf{y},\delta_{\varepsilon}\mathsf{y}\right)=\int d\delta_{\varepsilon}\theta'e^{-ik\delta_{\varepsilon}\theta'}\,\varepsilon^{-Q\left(y\right)}\left[\sum_{j=0}^{N}a_{2j}\left(\mathsf{y},\theta';t\right)\varepsilon^{2j}+\frac{\varepsilon^{2N+1}}{t^{Q\left(y\right)/2}}R_{N+1}\left(\mathsf{y},\theta';t\right)\right]\label{eq:rescaling exp. Bochner}
\end{equation}
uniformly in $t\leq1$ and $\mathsf{y}\in B_{R}\left(0\right)$, $\forall R>0$.
A slight difference above being that the coefficients $a_{j}\left(\mathsf{y},\theta';t\right)$
above are computed with respect to the model nilpotent sR Laplacian
$\hat{\Delta}_{y}\coloneqq\hat{\Delta}_{g_{y}^{TY},j^{r-2}R_{y}^{L}}$
\prettyref{eq:nilpotent Laplacian} on the product $S_{\theta}^{1}\times\mathbb{R}^{n-1}$
rather than \prettyref{eq: nilpotent Laplacian Euclidean} on Euclidean
space. In particular, the leading term is $a_{0}\left(\mathsf{y},\theta';t\right)=e^{-t\hat{\Delta}_{y}}\left(\mathsf{y},0;\mathsf{y},\theta'\right)$.
Now set $\varepsilon=k^{-\frac{1}{r}}$ and $r_{1}\left(y\right)\coloneqq1-\frac{r\left(y\right)}{r}$
to obtain 
\begin{align}
e^{-\frac{t}{k^{2/r}}\Delta_{k}}\left(k^{-\frac{1}{r}}\mathsf{y},k^{-\frac{1}{r}}\mathsf{y}\right) & =\int d\delta_{k^{-1/r}}\theta'e^{-ik^{r_{1}\left(y\right)}\theta'}\,k^{Q\left(y\right)/r}\left[\sum_{j=0}^{N}a_{2j}\left(\mathsf{y},\theta';t\right)k^{-2j/r}+O\left(k^{-\left(2N+1\right)/r}\right)\right]\nonumber \\
 & =\begin{cases}
k^{\left(n-1\right)/r}\left[\sum_{j=0}^{N}a_{2j}\left(\mathsf{y};t\right)k^{-2j/r}+O\left(k^{-\left(2N+1\right)/r}\right)\right]; & y\in Y_{r}\\
O\left(k^{-\infty}\right); & y\in Y_{\leq r-1}
\end{cases}\label{eq: Bochner heat kernel exp off. diag.}
\end{align}
following a stationary phase expansion in $\theta'$. Finally, setting
$y=y_{1}=\mathsf{y}$ in \prettyref{eq: Bochner heat kernel exp off. diag.}
proves the theorem.

Above we again note that the remainders are uniform for $\mathsf{y}\in B_{R}\left(0\right),\forall R>0$.
The first coefficient is given by the model Laplacian on the tangent
space $\Delta_{y}\coloneqq\Delta_{g_{y}^{TY},j^{r-2}R_{y}^{L}}$ via
\[
a_{0}\left(\mathsf{y};t\right)=\int d\theta'e^{-i\theta'}e^{-t\hat{\Delta}_{y}}\left(\mathsf{y},0;\mathsf{y},\theta'\right)=e^{-t\Delta_{y}}\left(\mathsf{y},\mathsf{y}\right)
\]
 by \prettyref{eq: model heat ker. reln.}. While the general coefficient
has the form
\begin{align}
a_{2j}\left(\mathsf{y};t\right) & =-\frac{1}{\pi}\int_{\mathbb{C}}\bar{\partial}\tilde{\rho}\left(z\right)\mathtt{C}_{2j}^{z}\left(\mathsf{y},\mathsf{y}\right)dzd\bar{z}\nonumber \\
\mathtt{C}_{2j}^{z} & =\sum_{p\in I_{2j}}\left(\Delta_{y}-z\right)^{-1}\left[\prod_{\alpha}\triangle_{p_{\alpha}}\left(\Delta_{y}-z\right)^{-1}\right]\label{eq: coeff. Bochner exp.}
\end{align}
as in \prettyref{eq:resolvent exp. terms}, for some set of second-order
differential operators $\triangle_{j}$, $j=1,2,\ldots\,$, (see also
\prettyref{eq:rescaling exp. Bochner} below). Above $\tilde{\rho}$
denotes an almost analytic continuation of $\rho\in\mathcal{S}\left(\mathbb{R}\right)$
satisfying $\rho\left(x\right)=e^{-tx}$, $x\geq0$. 
\end{proof}
We now show how the heat kernel expansion immediately proves our first
\prettyref{thm:First eigenvalue Bochner}.
\begin{proof}[Proof of \prettyref{thm:First eigenvalue Bochner}]
 We first give a short argument for asymptotic bounds on the smallest
eigenvalue
\begin{equation}
C_{1}k^{2/r}-C_{1}\leq\lambda_{0}\left(k\right)\leq C\left[1+o\left(1\right)\right]k^{2/r}.\label{eq:asymptotic bounds for l0}
\end{equation}
The upper bound follows easily from a min-max argument. Namely by
the min-max principle for self-adjoint operators applied to the model
operator $\Delta_{y}$ on the tangent space at $y\in Y_{r}$, there
exists $\tilde{\psi}\in C_{c}^{\infty}\left(\mathbb{R}^{n-1}\right)$,
$\left\Vert \psi\right\Vert =1$ such that $\left\langle \Delta_{y}\tilde{\psi},\tilde{\psi}\right\rangle \leq\lambda_{0}\left(\Delta_{y}\right)+\varepsilon,$
for each $\varepsilon>0$. Furthermore, the model operator arises
as the leading term $\left(\delta_{k^{-1/r}}\right)_{*}\Delta_{k}=k^{2/r}\left[\Delta_{y}+O\left(k^{-1/r}\right)\right]$
under the rescaling $\delta_{k^{-1/r}}y\coloneqq k^{-1/r}y$ in geodesic
coordinates centered at $y$ (cf. also subsection \prettyref{subsec:First-eigenvalue expansion}
below). From the min-max principle for $\Delta_{k}$ one then obtains
\[
\frac{\lambda_{0}\left(k\right)}{k^{2/r}}\leq k^{-2/r}\left\langle \Delta_{k}\tilde{\psi}_{0}^{k},\tilde{\psi}_{0}^{k}\right\rangle \leq\lambda_{0}\left(\Delta_{y}\right)+o\left(1\right)
\]
 for $\tilde{\psi}_{0}^{k}\coloneqq k^{\left(n-1\right)/r}\left(\delta_{k^{-1/r}}\right)^{*}\tilde{\psi}$.
The upper bound \prettyref{eq:asymptotic bounds for l0} now follows.
For the lower bound, we combine the trick of Guillemin-Uribe with
the Rothschild-Stein subelliptic estimate \prettyref{eq:local subelliptic estimate}
on the circle bundle to obtain 
\[
C_{1}\left\Vert \partial_{\theta}^{1/r}s\right\Vert ^{2}\leq C_{1}\left\Vert s\right\Vert _{H^{1/r}}^{2}\leq\left[\left\langle \Delta_{g^{E},F,\mu_{X}}s,s\right\rangle +\left\Vert s\right\Vert _{L^{2}}^{2}\right],
\]
$\forall s\in C^{\infty}\left(X;F\right).$ Letting $s=\pi^{*}\psi_{0}^{k}$
be the pullback of the orthonormal eigenfunction $\psi_{0}^{k}$ of
$\Delta_{k}$ gives $C_{1}k^{2/r}\leq\left(\lambda_{0}\left(k\right)+1\right)$
as required. To obtain the leading asymptotic \prettyref{eq: estimate first eigenvalue}
in \prettyref{thm:First eigenvalue Bochner} however one needs to
show $C_{1}=C$ in \prettyref{eq:asymptotic bounds for l0}. This
requires a closer look at the Rothschild-Stein subelliptic estimate
\prettyref{eq:local subelliptic estimate} and in particular identifying
the sharp constant therein. 

We instead take an alternate route via the heat kernel, this is also
consistent with our proofs of the other two theorems in the introduction.
First for any $0<t_{1}<t_{2}$, $y\in Y_{r}$ and $R>0$, one has
the following estimate at leading order using \prettyref{eq: Bochner heat kernel exp off. diag.}
\begin{align}
\frac{\lambda_{0}\left(k\right)}{k^{2/r}} & \leq\frac{1}{\left(t_{2}-t_{1}\right)}\ln\left(\frac{\int_{B_{R}\left(0\right)}d\left(k^{-\frac{1}{r}}\mathsf{y}\right)e^{-\frac{t_{1}}{k^{2/r}}\Delta_{k}}\left(k^{-\frac{1}{r}}\mathsf{y},k^{-\frac{1}{r}}\mathsf{y}\right)}{\int_{B_{R}\left(0\right)}d\left(k^{-\frac{1}{r}}\mathsf{y}\right)e^{-\frac{t_{2}}{k^{2/r}}\Delta_{k}}\left(k^{-\frac{1}{r}}\mathsf{y},k^{-\frac{1}{r}}\mathsf{y}\right)}\right)\nonumber \\
 & =\frac{1}{\left(t_{2}-t_{1}\right)}\ln\left(\frac{\int_{B_{R}\left(0\right)}d\mathsf{y}e^{-t_{1}\Delta_{y}}\left(\mathsf{y},\mathsf{y}\right)+O\left(k^{-1/r}\right)}{\int_{B_{R}\left(0\right)}d\mathsf{y}e^{-t_{2}\Delta_{y}}\left(\mathsf{y},\mathsf{y}\right)+O\left(k^{-1/r}\right)}\right)\nonumber \\
 & =\frac{1}{\left(t_{2}-t_{1}\right)}\ln\left(\frac{\int_{B_{R}\left(0\right)}d\mathsf{y}e^{-t_{1}\Delta_{y}}\left(\mathsf{y},\mathsf{y}\right)}{\int_{B_{R}\left(0\right)}d\mathsf{y}e^{-t_{1}\Delta_{y}}\left(\mathsf{y},\mathsf{y}\right)}\right)+O\left(k^{-1/r}\right).\label{eq: upper est. first ev.}
\end{align}
This already gives an upper bound on the first eigenvalue. To identify
the constant \prettyref{eq: upper constant} one takes the limit as
$t_{1}\rightarrow t_{2}$ to obtain 
\begin{align*}
\frac{\lambda_{0}\left(k\right)}{k^{2/r}} & \leq\frac{\int_{B_{R}\left(0\right)}d\mathsf{y}\left[\Delta_{y}e^{-t_{1}\Delta_{y}}\right]\left(\mathsf{y},\mathsf{y}\right)}{\int_{B_{R}\left(0\right)}d\mathsf{y}e^{-t_{1}\Delta_{y}}\left(\mathsf{y},\mathsf{y}\right)}+O\left(k^{-1/r}\right),
\end{align*}
$\forall t_{1}>0$. Using Proposition \prettyref{prop:large time heat =00003D000026 bottom of spec}
of \prettyref{sec:Model-operators}, this gives $\limsup_{k\rightarrow\infty}\frac{\lambda_{0}\left(k\right)}{k^{2/r}}\leq\lambda_{0}\left(\Delta_{y}\right)+\varepsilon,$$\forall\varepsilon>0$,
$y\in Y_{r}$, and hence 
\begin{equation}
\limsup_{k\rightarrow\infty}\frac{\lambda_{0}\left(k\right)}{k^{2/r}}\leq\inf_{y\in Y_{r}}\lambda_{0}\left(\Delta_{y}\right).\label{eq:upper bound lambda0}
\end{equation}
For the lower bound on $\lambda_{0}\left(k\right)$, first note that
as in \prettyref{eq:functional expansion.} one may prove an on diagonal
expansion 
\[
\varphi\left(\frac{1}{k^{2/r}}\Delta_{k}\right)\left(y,y\right)=k^{\left(n-1\right)/r}\left[a_{0}^{\varphi}\left(x\right)+a_{1}^{\varphi}\left(x\right)k^{-1/r}+\ldots+a_{N}^{\varphi}\left(x\right)k^{-N/r}+O\left(k^{-\left(N+1\right)/r}\right)\right]
\]
$\forall\varphi\in\mathcal{S}\left(\mathbb{R}\right)$, and where
the coefficient $a_{j}^{\varphi}$ has the form \prettyref{eq: coeff. Bochner exp.}
with $\tilde{\rho}$ replaced with an analytic continuation of $\varphi$.
Next note that each of the terms $\mathtt{C}_{2j}^{z}$ \prettyref{eq: coeff. Bochner exp.}
is holomorphic in $z$ for $\textrm{Re}z<C\coloneqq\inf_{y\in Y_{r}}\lambda_{0}\left(\Delta_{y}\right)$.
This gives $\varphi\left(\frac{1}{k^{2/r}}\Delta_{k}\right)\left(y,y\right)=O\left(k^{-N}\right),$$\forall N\in\mathbb{N},$
uniformly in $y\in Y$, for $\varphi\in C_{c}^{\infty}\left(-\infty,C\right)$.
Thus 
\begin{align}
\varphi\left(\frac{\lambda_{0}\left(k\right)}{k^{2/r}}\right) & \leq\textrm{tr }\varphi\left(\frac{1}{k^{2/r}}\Delta_{k}\right)=O\left(k^{-N}\right)\nonumber \\
\textrm{and hence}\qquad\qquad\qquad\inf_{y\in Y_{r}}\lambda_{0}\left(\Delta_{y}\right) & \leq\liminf_{k\rightarrow\infty}\frac{\lambda_{0}\left(k\right)}{k^{2/r}}.\label{eq:lower bound lambda0}
\end{align}
From \prettyref{eq:upper bound lambda0}, \prettyref{eq:lower bound lambda0}
we have \prettyref{eq: estimate first eigenvalue}. 

The estimate on the eigenfunction \prettyref{eq: estimate first eigenfunction}
then follows from $\left|\psi_{0}^{k}\left(y\right)\right|^{2}\leq e^{\frac{\lambda_{0}\left(k\right)}{k^{2/r}}}e^{-\frac{1}{k^{2/r}}\Delta_{k}}\left(y,y\right)$
on using \prettyref{eq: heat kernel expansion-1} and \prettyref{eq: estimate first eigenvalue}.
\end{proof}

\subsection{\label{subsec:Weyl-laws}Weyl law}

In this subsection and the next, we shall prove \prettyref{thm: Weyl law =00003D000026 expansion}
assuming $Y_{r}=\bigcup_{j=1}^{N}Y_{r,j}$ to be a union of embedded
submanifolds, of dimensions $d_{j}\coloneqq\textrm{dim}\left(Y_{r,j}\right)$,
along which the curvature $R^{L}$ vanishes non-degenerately \prettyref{eq:non-degeneracy assumption}.
Before proceeding, the following remark on our non-degeneracy hypothesis
is in order. 
\begin{rem}
\label{rem:(Non-degeneracy-hypothesis)} (Non-degeneracy hypothesis)
The non-degeneracy hypothesis \prettyref{eq:non-degeneracy assumption}
can be described more explicitly in local coordinates. Namely, if
we choose a coordinate system 
\[
\left(\underbrace{y_{1},\ldots,y_{d_{j}}}_{=y'};\underbrace{y_{d_{j}+1},\ldots,y_{n-1}}_{=y''}\right)
\]
 near $y\in Y_{r,j}$ in which $Y_{r,j}=\left\{ y''=0\right\} $ is
given by the vanishing of the last $n-1-d_{j}$ of these coordinates,
then the curvature can be Taylor expanded as
\begin{equation}
R^{L}=\underbrace{\sum_{\left|\alpha\right|=r-2}\sum_{p,q=1}^{n}R_{pq,\alpha}\left(y''\right)^{\alpha}dy_{p}dy_{q}}_{=R_{0}^{L}}+O\left(\left(y''\right)^{r-1}\right).\label{eq:curvature Taylor exp.}
\end{equation}
The non-degeneracy condition \prettyref{eq:non-degeneracy assumption}
is now seen to be equivalent to the implication
\begin{equation}
\left(\partial^{\beta}R_{0}^{L}\right)\left(y\right)=0,\,\:\forall\left|\beta\right|<r-2\,\iff y''=0.\label{eq:loc. non-deg. hyp}
\end{equation}
That is, the $\left(r-2\right)$-order vanishing locus $Y_{r}=Y_{r}^{0}\coloneqq\left\{ y\in Y|\textrm{ord}_{y}\left(R_{0}^{L}\right)=r-2\right\} $
is locally the same for the curvature $R^{L}$ and its leading part
$R_{0}^{L}$. An example of a curvature that is not non-degenerate
in this sense is $R^{L}=\left(y_{1}^{2}+y_{2}^{4}\right)dy_{1}dy_{2}$.
Here $r=4$, the leading part of the curvature is $R_{0}^{L}=y_{1}^{2}dy_{1}dy_{2}$,
while $\left\{ 0\right\} =Y_{4}\neq Y_{4}^{0}=\left\{ y_{1}=0\right\} $.

A more restrictive condition, that is common in the literature and
satisfied in the Montgomery case \cite{Montgomery-95paper}, is that
the curvature $R^{L}$ defines a 'magnetic well' at $Y_{r}$ \cite{Helffer-Mohamed96}.
This assumes the existence of positive constants $C_{1},C_{2}>0$
for which the curvature satisfies 
\begin{equation}
C_{1}d^{g}\left(y,Y_{r}\right)^{r-2}\leq\left|R^{L}\left(y\right)\right|\leq C_{2}d^{g}\left(y,Y_{r}\right)^{r-2},\quad\forall y\in Y,\label{eq:megnetic well condition}
\end{equation}
with $d^{g}$ denoting the Riemannian distance above. It is easy to
see that the above \prettyref{eq:megnetic well condition} is stronger
than and implies our non-degeneracy hypothesis \prettyref{eq:loc. non-deg. hyp}.
Examples of curvatures in dimension two that are non-degenerate \prettyref{eq:loc. non-deg. hyp}
without defining a magnetic well \prettyref{eq:megnetic well condition}
are $R^{L}=y_{1}y_{2}dy_{1}dy_{2}$ (normal crossing), $y_{1}y_{2}\left(y_{1}+y_{2}\right)dy_{1}dy_{2},\,y_{1}y_{2}\left(y_{1}^{2}-y_{2}^{2}\right)dy_{1}dy_{2}$
(multiple crossings), $y_{1}\left(y_{2}-y_{1}^{2}\right)dy_{1}dy_{2}$
(tangential crossing), $y_{1}\left(y_{1}^{3}-y_{2}^{2}\right)dy_{1}dy_{2}$
(cuspidal vanishing) and $y_{1}\left(y_{1}^{k+1}\pm y_{2}^{2}\right)dy_{1}dy_{2}$
($A_{k}^{\pm}$ singularity). While in higher dimension a general
class of examples is given by curvatures of the form $R^{L}=fdf\wedge y_{1}dy_{1}\in\Omega^{2}\left(\mathbb{R}^{n-1}\right)$,
for $f=y_{2}\ldots y_{n-1}g\left(y_{2},\ldots,y_{n-1}\right)$, with
$g$ being any homogeneous polynomial. The vanishing set $Y_{\geq3}$
for these curvatures includes $\left\{ 0\right\} \times V\left[g\right]$
for the variety $V\left[g\right]\coloneqq\left\{ g=0\right\} \subset\mathbb{R}_{y_{2},\ldots,y_{n}}^{n-2}$
corresponding to the arbitrary homogeneous polynomial $g$. While
the highest order vanishing locus $Y_{r}=\left\{ 0\right\} $ is the
origin for the above.
\end{rem}

By a standard Tauberian argument, the first part of \prettyref{thm: Weyl law =00003D000026 expansion}
on the asymptotics of the Weyl counting function now follows from
the following heat trace expansion. 
\begin{thm}
\label{thm: local Weyl law} Assume that $Y_{r}=\bigcup_{j=1}^{N}Y_{r,j}$
is a union of embedded submanifolds, of dimensions $d_{j}\coloneqq\textrm{dim}\left(Y_{r,j}\right)$,
along which the curvature $R^{L}$ vanishes non-degenerately \prettyref{eq:non-degeneracy assumption}.
For any $f\in C^{\infty}\left(Y\right)$, the heat trace of the Bochner
Laplacian satisfies the asymptotics 
\begin{equation}
\textrm{tr}\left[fe^{-\frac{t}{k^{2/r}}\Delta_{k}}\right]=\sum_{j=1}^{N}\left\{ \sum_{s=0}^{M}k^{\left(d_{j}-2s\right)/r}\left[\int_{NY_{r,j}}a_{j,s}\left(f;t\right)\right]+O\left(k^{\left(d_{j}-2M-1\right)/r}\right)\right\} \label{eq:trace expansion Bochner}
\end{equation}
$\forall M\in\mathbb{N}$, $t\leq1$. Moreover, the leading terms
above are given by 
\begin{equation}
a_{j,0}\left(f;t\right)=\left.f\right|_{Y_{r,j}}e^{-t\Delta_{y}}\left(v,v\right),\;\;v\in N_{y}Y_{r,j},\label{eq:leading term heat trace}
\end{equation}
in terms of the pullback to the normal bundle of $\left.f\right|_{Y_{r,j}}$. 
\end{thm}

\begin{proof}
By \prettyref{thm:Bochner heat kernel expansion} it suffices to consider
$f$ supported in a sufficiently small neighborhood of a given point
$y\in Y_{r,j}$. We then again choose a coordinate system near $y$
in which $Y_{r,j}$ is given by the vanishing of the last $n-1-d_{j}$
of the coordinates and in which the curvature has the Taylor expansion
\prettyref{eq:curvature Taylor exp.}. We may further assume the coordinate
vector fields $\left\{ \partial_{y_{j}}\right\} _{j=1}^{n-1}$ to
be orthonormal at $y.$ The model operator \prettyref{eq:model Bochner in coordinates}
on the tangent space 
\[
\Delta_{y}=-\sum_{\left|\alpha\right|=r-2}\sum_{p,q=1}^{n}\left(\partial_{y_{p}}+\frac{i}{r}y^{q}\left(y''\right)^{\alpha}R_{pq,\alpha}\right)^{2},
\]
is given in terms of this leading part of the curvature. Below it
shall also be useful to define the model semiclassical $\mathsf{k}$-Bochner
Laplacian 
\begin{equation}
\Delta_{y;\mathsf{k}}^{\textrm{mod}}\coloneqq-\sum_{\left|\alpha\right|=r-2}\sum_{p,q=1}^{n}\left(\partial_{y_{p}}+\frac{i\mathsf{k}}{r}y^{q}\left(y''\right)^{\alpha}R_{pq,\alpha}\right)^{2},\quad\forall\mathsf{k}>0,\label{eq:N model Bochner}
\end{equation}
corresponding to the leading part of the curvature in \prettyref{eq:curvature Taylor exp.}.

Next from \prettyref{eq: Bochner heat kernel exp off. diag.} one
has 
\begin{align}
e^{-\frac{1}{k^{2/r}}\Delta_{k}}\left(\delta_{\varepsilon}\mathsf{y},\delta_{\varepsilon}\mathsf{y}\right) & =k^{\left(n-1\right)/r}\left[\sum_{j=0}^{N}a_{2j}\left(\varepsilon k^{1/r}\mathsf{y};t\right)k^{-2j/r}+O\left(k^{-\left(2N+1\right)/r}\right)\right],\nonumber \\
\textrm{with }\quad a_{0}\left(\varepsilon k^{1/r}\mathsf{y};t\right) & =e^{-t\Delta_{y}}\left(\varepsilon k^{1/r}\mathsf{y},\varepsilon k^{1/r}\mathsf{y}\right),\label{eq:leading term near sing.}
\end{align}
uniformly for $k^{-1/r}\geq\varepsilon$ and $\mathsf{y}\in B_{1}\left(0\right)$.
Furthermore, substituting $t=\frac{1}{\varepsilon^{2}k^{2/r}}$ in
\prettyref{eq:rescaling exp. Bochner} we obtain 
\begin{align}
e^{-\frac{1}{k^{2/r}}\Delta_{k}}\left(\delta_{\varepsilon}\mathsf{y},\delta_{\varepsilon}\mathsf{y}\right) & =\varepsilon^{-\left(n-1\right)}\int d\theta'e^{-ik\varepsilon^{r}\theta'}\,\left[\sum_{j=0}^{N}a_{2j}\left(\mathsf{y},\theta';\frac{1}{\varepsilon^{2}k^{2/r}}\right)\varepsilon^{2j}\right.\label{eq:Bochner of diag heat kernel exp.}\\
 & \qquad\qquad\left.+\frac{\varepsilon^{2N+1}}{\left(\varepsilon k^{1/r}\right)^{n-1+r}}R_{2N+1}\left(\mathsf{y},\theta';\frac{1}{\varepsilon^{2}k^{2/r}}\right)\right],\nonumber \\
a_{0}\left(\mathsf{y},\theta';\frac{1}{\varepsilon^{2}k^{2/r}}\right) & =e^{-\frac{1}{\varepsilon^{2}k^{2/r}}\hat{\Delta}_{y}}\left(\mathsf{y},0;\mathsf{y},\theta'\right)\label{eq: leading heat kernel off sing.}
\end{align}
uniformly for $k\in\mathbb{N}$, $k^{-1/r}\leq\varepsilon$ and $\mathsf{y}\in B_{1}\left(0\right)$.
The leading term above is identified with the heat kernel 
\[
e^{-\frac{1}{\mathsf{k}^{2/r}}\Delta_{y;\mathsf{k}}^{\textrm{mod}}}\left(\mathsf{y},\mathsf{y}\right)=\int d\theta'e^{-ik\varepsilon^{r}\theta'}\,a_{0}\left(\mathsf{y},\theta';\frac{1}{\varepsilon^{2}k^{2/r}}\right),
\]
of the model $\mathsf{k}$-Bochner Laplacian \prettyref{eq:N model Bochner}
for $\mathsf{k}\coloneqq k\varepsilon^{r}$. One next chooses 
\[
\mathsf{y}=\Big(\underbrace{0,\ldots,0}_{=y'};\underbrace{y_{d_{j}+1},\ldots,y_{n-1}}_{=y''}\Big),\:\left|y''\right|=1,
\]
of the given form so that $\textrm{ord}_{\mathsf{y}}\left(R_{0}^{L}\right)<r-2$
by \prettyref{eq:loc. non-deg. hyp}. Then 
\begin{equation}
e^{-\frac{1}{\mathsf{k}^{2/r}}\Delta_{y;\mathsf{k}}^{\textrm{mod}}}\left(\mathsf{y},\mathsf{y}\right)=e^{-\Delta_{y}}\left(\mathsf{k}^{1/r}\mathsf{y},\mathsf{k}^{1/r}\mathsf{y}\right)=O\left(\mathsf{k}^{-\infty}\right),\label{eq: heat ker. int away sing.}
\end{equation}
follows by a stationary phase type argument as in \prettyref{thm:Bochner heat kernel expansion}.
A similar argument applied to the subsequent terms in \prettyref{eq: leading heat kernel off sing.},
which are given by convolution integrals with the leading part, shows
that $\int d\theta'e^{-ik\varepsilon^{r}\theta'}\,a_{2j}\left(\mathsf{y},\theta';\frac{1}{\varepsilon^{2}k^{2/r}}\right)=O\left(\mathsf{k}^{-\infty}\right)$,
$\forall j$. In particular, the terms of \prettyref{eq:Bochner of diag heat kernel exp.},
\prettyref{eq: leading heat kernel off sing.} are integrable in $\varepsilon$
for fixed $k$. Thus \prettyref{eq:leading term near sing.}, \prettyref{eq: leading heat kernel off sing.},
\prettyref{eq: heat ker. int away sing.} and a Taylor expansion for
$f$ near $\mathsf{y}=0$ combine to give \prettyref{eq:trace expansion Bochner}. 
\end{proof}

\subsection{\label{subsec:First-eigenvalue expansion}Expansion for the smallest
eigenvalue}

In this subsection we prove the second part of \prettyref{thm: Weyl law =00003D000026 expansion}
on the expansion for the first eigenvalue $\lambda_{0}\left(k\right)$,
assuming non-degeneracy \prettyref{eq:non-degeneracy assumption}
and when $Y_{r}$ is a finite set of points. The same argument as
below, with a minor modification, also gives an expansion for the
$m$th eigenvalue $\lambda_{m}\left(k\right)$ for any fixed $m\in\mathbb{N}_{0}$.

Before proceeding, we note a short argument showing that a weaker
version of the second part \prettyref{eq:expansion for small eigenvalues}
of \prettyref{thm: Weyl law =00003D000026 expansion} is immediate
from its first part \prettyref{eq:Weyl law}. Namely, when $Y_{r}$
is a finite set of points (or $d_{j}^{\max}=0$), the number of eigenvalues
for the Bochner Laplacian $N\left[c_{1}k^{2/r},c_{2}k^{2/r}\right]$,
for $c_{1}<C<c_{2}$, has a limit as $k\rightarrow\infty$ by \prettyref{eq:Weyl law}.
Furthermore, by \prettyref{thm: local Weyl law} the functional traces
$\textrm{ tr }\varphi\left(\frac{1}{k^{2/r}}\Delta_{k}\right)$, $\varphi\in C_{c}^{\infty}\left(c_{1},c_{2}\right)$,
involving the eigenvalues in this interval, have expansions in powers
of $k^{-1/r}$. Thus for $k$ sufficiently large $\lambda_{0}\left(k\right)$
is a root of the polynomial $p_{k^{-1/r}}\left(\lambda\right)\coloneqq\Pi_{\lambda_{j}\left(k\right)\in\left[c_{1}k^{2/r},c_{2}k^{2/r}\right]}\left(\lambda-\lambda_{j}\left(k\right)\right)$
of a fixed degree in $\lambda$. The coefficients of this polynomial
can be written in terms of the functional traces and hence have expansions
in powers of $k^{-1/r}$. By an application of analytic perturbation
theory for polynomial roots \cite[Ch. 2.2]{Kato}, the smallest eigenvalue
$\lambda_{0}\left(k\right)$ has an expansion in powers of $k^{-1/Mr}$,
where $M\in\mathbb{N}$ is the multiplicity of one of the roots of
$p_{0}\left(\lambda\right)$. 

The above argument is however insufficient to obtain an expansion
in powers of $k^{-1/r}$. Below we instead show that $\lambda_{0}\left(k\right)$
is an eigenvalue of a family of self-adjoint matrices $A_{k^{-1/r}}$,
of fixed rank, whose entries admit expansions in $k^{-1/r}$. One
may then apply analytic perturbation theory for self-adjoint matrices.
This requires working at the level of eigenfunctions and our technique
again partly borrows from \cite[Ch. 9]{Bismut-Lebeau91}.

We first need some terminology. Let $\varrho<\min\left\{ \frac{1}{2},\frac{1}{2}i_{g^{TY}}\right\} $
be smaller than half the injectivity radius $i_{g^{TY}}$ of $\left(Y,g^{TY}\right)$.
Choose a geodesic coordinate system on a ball $B_{2\varrho}\left(y\right)$
centered at $y\in Y_{r}$. Below it shall also be useful to choose
$\varrho$ small enough so that the balls $\left\{ B_{2\varrho}\left(y\right)\right\} _{y\in Y_{r}}$
are disjoint. Choose local trivializations $\mathtt{l}$, $\left\{ s_{j}\right\} _{j=1}^{\textrm{rank }\left(F\right)}$
of $L$, $F$ over $B_{2\varrho}\left(y\right)$ that are parallel
with respect to $\nabla^{L}$, $\nabla^{F}$ respectively along geodesics
starting at the origin. The Bochner Laplacian can be written in this
local frame and coordinates as $\Delta_{k}=\left(\nabla^{F\otimes L^{k}}\right)^{*}\nabla^{F\otimes L^{k}}$
where 
\begin{align}
\nabla^{F\otimes L^{k}}= & d+a^{F}+ka^{L}\nonumber \\
a_{p}^{L}= & \int_{0}^{1}d\rho\left(\rho y^{q}R_{pq}^{L}\left(\rho x\right)\right),\nonumber \\
a_{p}^{F}= & \int_{0}^{1}d\rho\left(\rho y^{q}R_{pq}^{F}\left(\rho x\right)\right),\label{eq:connection tensor product}
\end{align}
With $\chi\in C_{c}^{\infty}\left(\left[-1,1\right];\left[0,1\right]\right)$
with $\chi=1$ on $\left[-\frac{1}{2},\frac{1}{2}\right]$, we define
the modified connections on $\mathbb{R}^{n-1}$ via 
\begin{align}
\tilde{\nabla}^{F} & =d+\chi\left(\frac{\left|y\right|}{2\varrho}\right)a^{F}\nonumber \\
\tilde{\nabla}^{L} & =d+\left[\underbrace{\int_{0}^{1}d\rho\,\rho y^{k}\left(\tilde{R}^{L}\right)_{jk}\left(\rho y\right)}_{=\tilde{a}_{j}^{L}}\right]dy_{j},\quad\textrm{ where}\nonumber \\
\tilde{R}^{L} & =\chi\left(\frac{\left|y\right|}{2\varrho}\right)R^{L}+\left[1-\chi\left(\frac{\left|y\right|}{2\varrho}\right)\right]R_{0}^{L}.\label{eq:modified connection-1}
\end{align}
Further, we choose a modified metric $\tilde{g}^{TY}$ which is Euclidean
outside $B_{2\varrho}\left(y\right)$ and agrees with $g^{TY}$ on
$B_{\varrho}\left(y\right)$. This defines the modified Bochner Laplacian
$\tilde{\Delta}_{k}\coloneqq\left(\tilde{\nabla}^{F\otimes L^{k}}\right)^{*}\tilde{\nabla}^{F\otimes L^{k}}$
agreeing with $\Delta_{k}=\tilde{\Delta}_{k}$ on the geodesic ball
$B_{\varrho}\left(y\right)$.

A dilation as before is now defined via $\delta_{k^{-1/r}}y\coloneqq\left(k^{-1/r}y_{1},\ldots,k^{-1/r}y_{n-1}\right)$
and we consider the rescaled Bochner Laplacian 
\begin{equation}
\triangle\coloneqq k^{-2/r}\left(\delta_{k^{-1/r}}\right)_{*}\tilde{\Delta}_{k}.\label{eq:rescaled Bochner}
\end{equation}
Using a Taylor expansion and \prettyref{eq:modified connection-1},
the rescaled Bochner Laplacian has an expansion 
\begin{eqnarray}
\triangle & = & \left(\sum_{j=0}^{N}k^{-j/r}\triangle_{j}\right)+k^{-2\left(N+1\right)/r}\mathrm{E}_{N+1},\quad\forall N.\label{eq: Taylor expansion Bochner}\\
\textrm{where each }\;\triangle_{j} & = & a_{j;pq}\left(y\right)\partial_{y_{p}}\partial_{y_{q}}+b_{j;p}\left(y\right)\partial_{y_{p}}+c_{j}\left(y\right)\label{eq:operators in Bochner expansion}
\end{eqnarray}
is a $k$-independent, self-adjoint, second-order differential operator
while each 
\begin{equation}
\mathrm{E}_{j}=\sum_{\left|\alpha\right|=N+1}y^{\alpha}\left[a_{j;pq}^{\alpha}\left(y;k\right)\partial_{y_{p}}\partial_{y_{q}}+b_{j;p}^{\alpha}\left(y;k\right)\partial_{y_{p}}+c_{j}^{\alpha}\left(y;k\right)\right]\label{eq: error operators in Bochner}
\end{equation}
is a $k$-dependent self-adjoint, second-order differential operator
on $\mathbb{R}^{n-1}$ . Furthermore the functions appearing in \prettyref{eq:operators in Bochner expansion}
are polynomials with degrees satisfying 
\begin{align}
\textrm{deg }a_{j}=j,\quad\textrm{ deg }b_{j}\leq j+r-1,\quad & \textrm{deg }c_{j}\leq j+2r-2\nonumber \\
\textrm{ deg }b_{j}-\left(j-1\right)=\textrm{deg }c_{j}-j=0 & \quad(\textrm{mod }2)\label{eq:inequalities and restrictions on coefficients}
\end{align}
and whose coefficients involve 
\begin{align}
a_{j}: & \;\textrm{ atmost }j-2\textrm{ derivatives of }R^{TY}\nonumber \\
b_{j},c_{j}: & \;\textrm{ atmost }j-2\textrm{ derivatives of }R^{F},R^{TY}\textrm{ and atmost }j+r-2\textrm{ derivatives of }R^{L}\label{eq:number of derivatives in coefficients}
\end{align}
The coefficients $a_{j;pq}^{\alpha}\left(y;k\right),b_{j;p}^{\alpha}\left(y;k\right),c_{j}^{\alpha}\left(y;k\right)$
of \prettyref{eq: error operators in Bochner} are moreover uniformly
$C^{\infty}$ bounded in $k$. The leading term of \prettyref{eq: Taylor expansion Dirac}
is computed 
\begin{equation}
\triangle_{0}=\Delta_{y}\coloneqq\Delta_{g_{y}^{TY},j^{r-2}R_{y}^{L}}\label{eq:leading term Bochner}
\end{equation}
in terms of the model Bochner Laplacian on the tangent space $TY$
\prettyref{eq:model Bochner Laplacian}. We shall see below that these
operators \prettyref{eq:operators in Bochner expansion} are the same
as those appearing in \prettyref{eq: coeff. Bochner exp.}.

Next, in our chosen coordinates and trivialization, the curvature
$R^{L}$ again has a Taylor expansion as in \prettyref{eq:curvature Taylor exp.}
with the non-degeneracy condition \prettyref{eq:non-degeneracy assumption}
being equivalent to \prettyref{eq:loc. non-deg. hyp}. If $Y_{r}$
is further a finite set of points then the model operator \prettyref{eq:leading term Bochner}
at $y\in Y_{r}$ has a discrete spectrum, $\textrm{EssSpec}\left(\Delta_{y}\right)=\emptyset$,
by Proposition \prettyref{prop:model spectra} in \prettyref{sec:Model-operators}.
We then set $\lambda_{0,y}<\lambda_{1,y}$ to be the two smallest
eigenvalues of $\Delta_{y}$ and $E_{0,y}\coloneqq\ker\left[\Delta_{y}-\lambda_{0,y}\right]$
the smallest eigenspace. Any normalized $\tilde{\psi}\in E_{0,y}$
defines a quasimode 
\begin{align}
\widetilde{\psi}_{k}\left(y\right)\coloneqq & \chi\left(\frac{2\left|y\right|}{\varrho}\right)\underbrace{k^{\left(n-1\right)/2r}\tilde{\psi}\left(k^{1/r}y\right)}_{=k^{\left(n-1\right)/2r}\delta_{k^{-1/r}}^{*}\tilde{\psi}}\in C^{\infty}\left(Y;F\otimes L^{k}\right),\quad\textrm{ satisfying}\nonumber \\
\left\Vert \widetilde{\psi}_{k}\right\Vert = & 1+o\left(1\right)\nonumber \\
\Delta_{k}\widetilde{\psi}_{k}= & k^{2/r}\lambda_{0,y}\widetilde{\psi}_{k}+O_{L^{2}}\left(k^{1/r}\right).\label{eq:quasimodes =00003D000026 properties}
\end{align}
And we define $\tilde{E}_{0,y}$ to be the span of the quasimodes
corresponding to an orthonormal basis of $E_{0,y}$ . Finally set
$\bar{\lambda}_{0}\coloneqq\min_{y\in Y_{r}}\lambda_{0,y}$, $\bar{Y}_{r}\coloneqq\left\{ y\in Y_{r}|\lambda_{0,y}=\bar{\lambda}_{0}\right\} \subset Y_{r}$
and $\bar{\lambda}_{1}\coloneqq\min\left\{ \lambda_{1,y}|y\in\bar{Y}_{r}\right\} \cup\left\{ \lambda_{0,y}|y\in Y_{r}\setminus\bar{Y}_{r}\right\} >\bar{\lambda}_{0}$.
Further set $\tilde{E}_{0}\coloneqq\oplus_{y\in\bar{Y}_{r}}\tilde{E}_{0,y}\subset C^{\infty}\left(Y;F\otimes L^{k}\right)$
and $\tilde{E}_{0}^{\perp}$ to be its $L^{2}$ orthogonal complement. 

We now have the following proposition.
\begin{prop}
There exist $c>0$, $k_{0}\in\mathbb{N}$ such that 
\begin{align}
\left|\left\langle \Delta_{k}\tilde{\psi},\tilde{\psi}\right\rangle -\bar{\lambda}_{0}k^{2/r}\right| & \leq ck^{1/r}\label{eq: small quasi eigenspaces}\\
\left\langle \Delta_{k}\psi,\psi\right\rangle  & \geq\frac{1}{2}\left(\bar{\lambda}_{0}+\bar{\lambda}_{1}\right)k^{2/r}\label{eq: complement space}
\end{align}
for each $k>k_{0}$ and $\tilde{\psi}\in\tilde{E}_{0},\,\psi\in C^{\infty}\left(Y;F\otimes L^{k}\right)\cap\tilde{E}_{0}^{\perp}$
of unit $L^{2}$-norm. 
\end{prop}

\begin{proof}
The first equation \prettyref{eq: small quasi eigenspaces} follows
easily from construction \prettyref{eq:quasimodes =00003D000026 properties}.

For \prettyref{eq: complement space}, we first set $\chi_{y}\psi\coloneqq\chi\left(\frac{d^{g^{TY}}\left(.,y\right)}{\varrho}\right)\psi$,
with $d^{g^{TY}}$ being the Riemannian distance, for each $y\in Y_{r}$
and split 
\[
\psi=\underbrace{\left(\sum_{y\in Y_{r}}\chi_{y}\right)\psi}_{=\psi_{1}}+\underbrace{\left(1-\sum_{y\in Y_{r}}\chi_{y}\right)\psi}_{=\psi_{2}}.
\]
Now since the $\psi_{2}$ is compactly supported away from $Y_{r}$,
an argument similar to \prettyref{eq:asymptotic bounds for l0} gives
\begin{equation}
\left\langle \Delta_{k}\psi_{2},\psi_{2}\right\rangle \geq\left[c_{1}k^{2/\left(r-1\right)}-c_{2}\right]\left\Vert \psi_{2}\right\Vert ^{2}\label{eq: spectral estimate away from Yr}
\end{equation}
for some constants $c_{1},c_{2}>0$ depending only on $\varrho$.
Next since $\chi_{y}\psi$, $y\in\bar{Y}_{r}$, has compact support
in $B_{\varrho}\left(y\right)$, we may decompose 
\[
k^{-\left(n-1\right)/2r}\left(\delta_{k^{-1/r}}^{-1}\right)^{*}\chi_{y}\psi=\underbrace{\psi_{y}^{0}}_{\in\ker\left[\triangle_{0}-\bar{\lambda}_{0}\right]}+\underbrace{\psi_{y}^{+}}_{\in\ker\left[\triangle_{0}-\bar{\lambda}_{0}\right]^{\perp}}.
\]
Clearly $\psi_{y}^{0}$ is orthogonal to $\psi_{y}^{+}$ and $\triangle_{0}\psi_{y}^{+}$
while $\left\langle \triangle_{0}\psi_{y}^{+},\psi_{y}^{+}\right\rangle \geq\bar{\lambda}_{1}\left\Vert \psi_{y}^{+}\right\Vert ^{2}$
by definition. Furthermore, $\chi_{y}\psi\perp\tilde{E}_{0,y}$ by
construction and hypothesis. Hence we may compute 
\begin{align*}
\left\langle \chi_{y}\psi,k^{\left(n-1\right)/2r}\delta_{k^{-1/r}}^{*}\tilde{\psi}\right\rangle  & =\left\langle \chi_{y}\psi,\left(1-\chi\right)k^{\left(n-1\right)/2r}\delta_{k^{-1/r}}^{*}\tilde{\psi}\right\rangle \\
 & =\left\langle k^{-\left(n-1\right)/2r}\left(\delta_{k^{-1/r}}^{-1}\right)^{*}\chi_{y}\psi,\left[1-\chi\left(k^{-1/r}y\right)\right]\tilde{\psi}\right\rangle =o\left(1\right)\left\Vert \chi_{y}\psi\right\Vert 
\end{align*}
for any normalized $\tilde{\psi}\in E_{0,y}$. This in turn gives
$\left\Vert \psi_{y}^{0}\right\Vert =o\left(1\right)\left\Vert \chi_{y}\psi\right\Vert $,
$\left\Vert \psi_{y}^{+}\right\Vert =\left[1-o\left(1\right)\right]\left\Vert \chi_{y}\psi\right\Vert $
and hence 
\begin{align*}
\left\langle \triangle_{0}k^{-n/2r}\left(\delta_{k^{-1/r}}^{-1}\right)^{*}\chi_{y}\psi,k^{-n/2r}\left(\delta_{k^{-1/r}}^{-1}\right)^{*}\chi_{y}\psi\right\rangle  & =\left\langle \triangle_{0}\psi_{y}^{0},\psi_{y}^{0}\right\rangle +\left\langle \triangle_{0}\psi_{y}^{+},\psi_{y}^{+}\right\rangle \\
 & \geq\bar{\lambda}_{1}\left\Vert \psi_{y}^{+}\right\Vert ^{2}\geq\left[\bar{\lambda}_{1}-o\left(1\right)\right]\left\Vert \chi_{y}\psi\right\Vert ^{2}.
\end{align*}
On account of the rescaling \prettyref{eq:rescaled Bochner}, \prettyref{eq: Taylor expansion Bochner},
\prettyref{eq:leading term Bochner} we then have 
\begin{equation}
\left\langle \Delta_{k}\chi_{y}\psi,\chi_{y}\psi\right\rangle \geq k^{2/r}\left[\bar{\lambda}_{1}-o\left(1\right)\right]\left\Vert \chi_{y}\psi\right\Vert ^{2}.\label{eq: spectral estimate near Yr}
\end{equation}
Finally, with $\chi_{1}=\sum_{y\in\bar{Y}_{r}}\chi_{y}$ we estimate
\begin{align*}
\left\Vert \nabla^{F\otimes L^{k}}\psi\right\Vert  & \geq\rho\left\Vert \chi_{1}\nabla^{F\otimes L^{k}}\psi\right\Vert +\left(1-\rho\right)\left\Vert \left(1-\chi_{1}\right)\nabla^{F\otimes L^{k}}\psi\right\Vert \\
 & =\rho\left\Vert -d\chi_{1}\psi+\nabla^{F\otimes L^{k}}\chi_{1}\psi\right\Vert +\left(1-\rho\right)\left\Vert d\chi_{1}\psi+\nabla^{F\otimes L^{k}}\left(1-\chi_{1}\right)\psi\right\Vert \\
 & =\rho\left\Vert \nabla^{F\otimes L^{k}}\chi_{1}\psi\right\Vert +\left(1-\rho\right)\left\Vert \nabla^{F\otimes L^{k}}\left(1-\chi_{1}\right)\psi\right\Vert -O\left(1\right)\left\Vert \psi\right\Vert \\
 & \geq\rho k^{1/r}\left[\bar{\lambda}_{1}-o\left(1\right)\right]^{1/2}\left\Vert \chi_{1}\psi\right\Vert +\left(1-\rho\right)\left[c_{1}k^{2/\left(r-1\right)}-c_{2}\right]^{1/2}\left\Vert \left(1-\chi_{1}\right)\psi\right\Vert -O\left(1\right)\left\Vert \psi\right\Vert \\
 & \geq\frac{1}{2}\left(\bar{\lambda}_{0}+\bar{\lambda}_{1}\right)^{1/2}k^{1/r}\left\Vert \psi\right\Vert 
\end{align*}
for $k\gg0$ by \prettyref{eq: spectral estimate away from Yr} and
\prettyref{eq: spectral estimate near Yr}. 
\end{proof}
Following the above proposition, the min-max principle for eigenvalues
gives 
\begin{align}
\textrm{Spec}\left(\Delta_{k}\right) & \subset\underbrace{\left[\bar{\lambda}_{0}k^{2/r}-ck^{1/r},\bar{\lambda}_{0}k^{2/r}+ck^{1/r}\right]}_{\mathcal{I}_{k}\coloneqq}\cup\left[\frac{1}{2}\left(\bar{\lambda}_{0}+\bar{\lambda}_{1}\right)k^{2/r},\infty\right).\label{eq:spectral gap smallest ev.}
\end{align}
Next, choose $\alpha\in\left(\bar{\lambda}_{0},\frac{\bar{\lambda}_{0}+\bar{\lambda}_{1}}{2}\right)$.
And let $\Gamma=\left\{ \left|z\right|=\alpha\right\} $ and $\varphi\in C_{c}\left(0,\alpha\right)$,
with $\varphi=1$ near $\bar{\lambda}_{0}$, define a circular contour
in the complex plane and a cutoff function respectively. The resolvent
$\left(\frac{1}{k^{2/r}}\Delta_{k}-z\right)^{-1}$ then exists for
$z\in\Gamma$, $k\gg0$ and one may define via 
\[
P_{0}\coloneqq\frac{1}{2\pi i}\int_{\Gamma}\left(\frac{1}{k^{2/r}}\Delta_{k}-z\right)^{-1}=\varphi\left(\frac{1}{k^{2/r}}\Delta_{k}\right)
\]
the spectral projection onto the span of the $\Delta_{k}$-eigenspaces
with eigenvalue in the first interval $\mathcal{I}_{k}$ of \prettyref{eq:spectral gap smallest ev.}.
Finally, \prettyref{eq: small quasi eigenspaces} and \prettyref{eq: complement space}
imply that 
\begin{equation}
P_{0}:\tilde{E}_{0}\xrightarrow{\sim}E_{0}\coloneqq\bigoplus\Big\{\ker\left(\Delta_{k}-\lambda\right):\lambda\in\mathcal{I}_{k}\Big\}\label{eq: proj is iso.}
\end{equation}
is an isomorphism for $k\gg0$. We now have the following. 
\begin{thm}
\label{thm:quasimode MC expansion} For any two quasimodes $\widetilde{\psi}_{k},\widetilde{\psi}_{k}'\in\tilde{E}_{0}$
\prettyref{eq:quasimodes =00003D000026 properties}, the inner product
\begin{equation}
\left\langle \widetilde{\psi}_{k},\Delta_{k}P_{0}
\widetilde{\psi}_{k}'\right\rangle =
k^{2/r}\sum_{j=0}^{N}\tilde{c}_{j}k^{-j/r}+
O\left(k^{\left(1-N\right)/r}\right)\label{eq: IP expansion}
\end{equation}
has an asymptotic expansion for some $\tilde{c}_{j}\in\mathbb{R}$,
$j=0,1,\ldots$. 
\end{thm}

\begin{proof}
For two quasimodes $\widetilde{\psi}_{k},\widetilde{\psi}_{k}'$ localized
at two different points of $\bar{Y}_{r}$ one has $\left\langle \widetilde{\psi}_{k},\Delta_{k}P_{0}\widetilde{\psi}_{k}'\right\rangle =O\left(k^{-\infty}\right)$
following a similar off-diagonal decay for the kernel of $\varphi\left(\frac{1}{k^{2/r}}\Delta_{k}\right)$
as \prettyref{eq: off diagonal decay Bochner}. We now consider two
$\widetilde{\psi}_{k},\widetilde{\psi}_{k}'\in\tilde{E}_{0,y}$ of
the form \prettyref{eq:quasimodes =00003D000026 properties} localized
at the same point $y\in\bar{Y}_{r}$. In this case, first a finite
propagation argument as in \prettyref{eq: heat localization on diagonal}
gives
\begin{align}
\left\langle \widetilde{\psi}_{k},\Delta_{k}P_{0}\widetilde{\psi}_{k}'\right\rangle  & =\left\langle \widetilde{\psi}_{k},\tilde{\Delta}_{k}\varphi\left(\frac{1}{k^{2/r}}\tilde{\Delta}_{k}\right)\widetilde{\psi}_{k}'\right\rangle +O\left(k^{-\infty}\right),\quad\textrm{ while }\nonumber \\
\frac{1}{k^{2/r}}\tilde{\Delta}_{k}\varphi\left(\frac{1}{k^{2/r}}\tilde{\Delta}_{k}\right)\left(y,y'\right) & =k^{\left(n-1\right)/r}\triangle\varphi\left(\triangle\right)\left(k^{1/r}y,k^{1/r}y'\right)\label{eq: matrix coefficient in terms of model}
\end{align}
follows by a similar rescaling as in \prettyref{eq:coordinate dilation}.
We now obtain an expansion for the right-hand side above by a resolvent
expansion for $\triangle$ similar to \prettyref{eq:sR resolvent expansion}.
Namely, let 
\[
I_{j}\coloneqq\left\{ p=\left(p_{0},p_{1},\ldots\right)|p_{\alpha}\in\mathbb{N},\sum p_{\alpha}=j\right\} 
\]
 denote the set of partitions of the integer $j$ and define 
\[
\mathtt{C}_{j}^{z}\coloneqq\sum_{p\in I_{j}}\left(\triangle_{0}-z\right)^{-1}\left[\prod_{\alpha}\triangle_{p_{\alpha}}\left(\triangle_{0}-z\right)^{-1}\right].
\]
Then by repeated applications of the local elliptic estimate we have
\begin{equation}
\left(\triangle-z\right)^{-1}-\sum_{j=0}^{N}k^{-j/r}\mathtt{C}_{j}^{z}=O_{H_{\textrm{loc}}^{s}\rightarrow H_{\textrm{loc}}^{s+2}}\left(k^{-\left(N+1\right)/r}\left|\textrm{Im}z\right|^{-2rN-2}\right),\label{eq:resolvent expansion Bochner resolvent}
\end{equation}
for each $N\in\mathbb{N},\,s\in\mathbb{R}$. Plugging the above expansion
into the Helffer-Sjöstrand formula then gives 
\begin{equation}
\triangle\varphi\left(\triangle\right)-\sum_{j=0}^{N}k^{-j/r}\mathtt{C}_{j}^{\varphi}=O_{H_{\textrm{loc}}^{s}\rightarrow H_{\textrm{loc}}^{s+2}}\left(k^{-\left(N+1\right)/r}\right)\label{eq:Bochner functional expansion}
\end{equation}
$\forall N\in\mathbb{N}$ and for some $k$-independent $\mathtt{C}_{j}^{\varphi}\in L^{2}\left(\mathbb{R}^{n-1}\times\mathbb{R}^{n-1}\right)$,
$j=0,1,\ldots$\,. A similar argument as \prettyref{eq: regularized expansion},
replacing \prettyref{eq:resolvent expansion Bochner resolvent} by
the resolvent expansion for $\left(\triangle+1\right)^{-M}\left(\triangle-z\right)^{-1}$,
shows that the last expansion above is valid in $C^{l}\left(\mathbb{R}^{n-1}\times\mathbb{R}^{n-1}\right)$,
$\forall l\in\mathbb{N}$. Hence plugging \prettyref{eq:Bochner functional expansion}
into \prettyref{eq: matrix coefficient in terms of model} finally
gives 
\begin{align*}
\left\langle \widetilde{\psi}_{k},\Delta_{k}P_{0}\widetilde{\psi}_{k}'\right\rangle -k^{2/r}\left(\sum_{j=0}^{N}c_{j}k^{-j/r}\right) & =O\left(k^{-\left(N-1\right)/r}\right)
\end{align*}
$\forall N\in\mathbb{N}$, with $\tilde{c}_{j}\coloneqq\left\langle \tilde{\psi},\mathtt{C}_{j}^{\varphi}\tilde{\psi}'\right\rangle $
as required. 
\end{proof}
The proof of \prettyref{thm: Weyl law =00003D000026 expansion} now
follows from the above and is summarized below.
\begin{proof}[Proof of \prettyref{thm: Weyl law =00003D000026 expansion}]
 As noted before, the first part of the theorem regarding the Weyl
law \prettyref{eq:Weyl law} follows from \prettyref{thm: local Weyl law}
by a Tauberian argument. 

For the second part of the theorem regarding the expansion for $\lambda_{0}\left(k\right)$,
note from \prettyref{eq: proj is iso.} that the low lying eigenvalues
of $\Delta_{k}$ are given by $\textrm{Spec}\left(\left.\Delta_{k}\right|_{E_{0}}\right)=\textrm{Spec}\left(\left.\Delta_{k}\right|_{P_{0}\tilde{E}_{0}}\right)$
for $k\gg0$. But since the matrix coefficients of $\left.\Delta_{k}\right|_{P_{0}\tilde{E}_{0}}$
were just shown to have an expansion in \prettyref{thm:quasimode MC expansion},
the expansion for the smallest eigenvalue $\lambda_{0}\left(k\right)$
now follows by an application of standard perturbation theory for
self-adjoint matrices as in \cite[Ch. 2.6]{Kato}.
\end{proof}
\begin{rem}
(Spectrum and abnormals) Our \prettyref{thm:First eigenvalue Bochner}
and \prettyref{thm: Weyl law =00003D000026 expansion} proved in this
section are generalizations of the results in \cite{Montgomery-95paper}.
This latter article seems to have been motivated by describing a correspondence
between the asymptotics of sR Laplace spectrum and the phenomenon
of singular or abnormal geodesies in sR geometry, and claims to have
achieved this goal. However our generalization of its results here
shows that this is not the case, as indeed the concentration of the
eigenfunction ultimately occurs on the locus $Y_{r}$ where the Hausdorff
dimension is maximized. And this in general has little if anything
to do with abnormals. As a reference for the first spectral study
of abnormals in sR geometry we refer instead to the recent article
\cite{Savale-QC} of the second author.
\end{rem}

\section{Kodaira Laplacian on tensor powers\label{subsec:Kodaira-Laplacian}}

In this final section, we shall prove the Bergman kernel expansion
in \prettyref{thm:Bergman kernel expansion}. Thus we now specialize
to the case when $Y$ is a complex Hermitian manifold with integrable
complex structure $J$. For the arguments of this section, we shall
further need to restrict to the two dimensional case, when $Y$ is
a Riemann surface (see Remark \prettyref{rem:good remark}). The metric
$g^{TY}$ is induced from the Hermitian metric on the complex tangent
space $T_{\mathbb{C}}Y=T^{1,0}Y$. Further $(L,h^{L})$, $(F,h^{F})$
are chosen to be a Hermitian, holomorphic line and vector bundles
respectively. We denote by $\nabla^{L}$, $\nabla^{F}$ the corresponding
Chern connections. The curvature $R^{L}$ of $\nabla^{L}$ is a $\left(1,1\right)$
form which is further assumed to be semipositive 
\begin{align}
R^{L}\left(w,\bar{w}\right) & \geq0,\quad\forall w\in T^{1,0}Y.\label{eq:semi-positivity-1}
\end{align}
We also assume as before that the curvature $R^{L}$ vanishes at finite
order at any point of $Y$. We note that semipositivity implies that
the order of vanishing $r_{y}-2\in2\mathbb{N}_{0}$ of the curvature
$R^{L}$ at any point $y$ is even. Semipositivity and finite order
of vanishing imply that there are points where the curvature is positive,
the set where the curvature is positive is in fact an open dense set.
Hence $\deg L=\int_{Y}\frac{i}{2\pi}R^{L}>0,$ so that $L$ is ample.

Denote by $\left(\Omega^{0,*}\left(X;F\otimes L^{k}\right);\bar{\partial}_{k}\right)$
the Dolbeault complex and define the Kodaira Laplace and Dirac operators
acting on $\Omega^{0,*}\left(X;F\otimes L^{k}\right)$ via
\begin{align}
\Box_{k}\coloneqq\frac{1}{2}\left(D_{k}\right)^{2} & =\bar{\partial}_{k}\bar{\partial}_{k}^{*}+\bar{\partial}_{k}^{*}\bar{\partial}_{k}\label{eq: Kodaira Laplace}\\
D_{k} & \coloneqq\sqrt{2}\left(\bar{\partial}_{k}+\bar{\partial}_{k}^{*}\right).\label{eq: Kodaira Dirac}
\end{align}

Clearly, $D_{k}$ interchanges while $\Box_{k}$ preserves $\Omega^{0,0/1}$.
We denote $D_{k}^{\pm}=\left.D_{k}\right|_{\Omega^{0,0/1}}$ and $\Box_{k}^{0/1}=\left.\Box_{k}\right|_{\Omega^{0,0/1}}$.
The Clifford multiplication endomorphism $c:TY\rightarrow\textrm{End}\left(\Lambda^{0,*}\right)$
is defined via $c\left(v\right)\coloneqq\sqrt{2}\left(v^{1,0}\wedge-i_{v^{0,1}}\right)$,
$v\in TY$, and extended multiplicatively to the entire exterior algebra
$\Lambda^{*}TY$.

Denote by $\nabla^{TY},\nabla^{T^{1,0}Y}$ the Levi-Civita and Chern
connections on the real and holomorphic tangent spaces as well as
by $\nabla^{T^{0,1}Y}$ the induced connection on the anti-holomorphic
tangent space. Denote by $\Theta$ the real $\left(1,1\right)$ form
defined by contraction of the complex structure with the metric $\Theta\left(.,.\right)=g^{TY}\left(J.,.\right)$.
This is clearly closed $d\Theta=0$, or $Y$ is Kähler, and the complex
structure is parallel $\nabla^{TY}J=0$ or $\nabla^{TY}=\nabla^{T^{1,0}Y}\oplus\nabla^{T^{1,0}Y}$.

With the induced tensor product connection on $\Lambda^{0,*}\otimes F\otimes L^{k}$
being denoted via $\nabla^{\Lambda^{0,*}\otimes F\otimes L^{k}}$,
the Kodaira Dirac operator \prettyref{eq: Kodaira Dirac} is now given
by the formula 
\begin{equation}
D_{k}=c\circ\nabla^{\Lambda^{0,*}\otimes F\otimes L^{k}}.\label{eq:spin-c Dirac}
\end{equation}

Next, we denote by $R^{F}$ the curvature of $\nabla^{F}$ and by
$\kappa$ the scalar curvature of $g^{TY}$. Define the following
endomorphisms of $\Lambda^{0,*}$ 
\begin{align}
\omega\left(R^{F}\right)\coloneqq R^{F}\left(w,\bar{w}\right)\bar{w}i_{\bar{w}}, & \qquad\quad\tau^{F}\coloneqq R^{F}\left(w,\bar{w}\right)\nonumber \\
\omega\left(R^{L}\right)\coloneqq R^{L}\left(w,\bar{w}\right)\bar{w}i_{\bar{w}}, & \qquad\quad\tau^{L}\coloneqq R^{L}\left(w,\bar{w}\right)\nonumber \\
\omega\left(\kappa\right)\coloneqq\kappa\bar{w}i_{\bar{w}},\qquad\quad\quad\label{eq:formulas Clifford-1}
\end{align}
in terms of an orthonormal section $w$ of $T^{1,0}Y$. The Lichnerowicz
formula for the above Dirac operator \cite[Thm. 1.4.7]{Ma-Marinescu}
simplifies for a Riemann surface and is given by 
\begin{align}
2\Box_{k} & =D_{k}^{2}=\left(\nabla^{\Lambda^{0,*}\otimes F\otimes L^{k}}\right)^{*}\nabla^{\Lambda^{0,*}\otimes F\otimes L^{k}}+k\left[2\omega\left(R^{L}\right)-\tau^{L}\right]+\left[2\omega\left(R^{F}\right)-\tau^{F}\right]+\frac{1}{2}\omega\left(\kappa\right).\label{eq:Lichnerowicz for.}
\end{align}

We now have the following. 
\begin{prop}
\label{prop:Dk spectral estimate} Let $Y$ be a compact Riemann surface,
$(L,h^{L})\to Y$ a semipositive line bundle whose curvature $R^{L}$
vanishes to finite order at any point. Let $(F,h^{F})\to Y$ be a
Hermitian holomorphic vector bundle. Then there exist constants $c_{1},c_{2}>0$,
such that 
\[
\left\Vert D_{k}s\right\Vert ^{2}\geq\left(c_{1}k^{2/r}-c_{2}\right)\left\Vert s\right\Vert ^{2}
\]
for all $s\in\Omega^{0,1}\left(Y;F\otimes L^{k}\right)$. 
\end{prop}

\begin{proof}
Writing $s=\left|s\right|\bar{w}\in\Omega^{0,1}\left(Y;F\otimes L^{k}\right)$
in terms of a local orthonormal section $\bar{w}$ gives 
\begin{equation}
\left\langle \left[2\omega\left(R^{L}\right)-\tau^{L}\right]s,s\right\rangle =R^{L}\left(w,\bar{w}\right)\left|s\right|^{2}\geq0\label{eq: semipositivity curvature estimate}
\end{equation}
from \prettyref{eq:semi-positivity-1}, \prettyref{eq:formulas Clifford-1}.
This gives 
\begin{align*}
\left\langle D_{k}^{2}s,s\right\rangle  & =\left\langle \left[\left(\nabla^{\Lambda^{0,*}\otimes F\otimes L^{k}}\right)^{*}\nabla^{\Lambda^{0,*}\otimes F\otimes L^{k}}+k\left[2\omega\left(R^{L}\right)-\tau^{L}\right]+\left[2\omega\left(R^{F}\right)-\tau^{F}\right]+\frac{1}{2}\omega\left(\kappa\right)\right]s,s\right\rangle \\
 & \geq\left\langle \left(\nabla^{\Lambda^{0,*}\otimes F\otimes L^{k}}\right)^{*}\nabla^{\Lambda^{0,*}\otimes F\otimes L^{k}}s,s\right\rangle -c_{0}\left\Vert s\right\Vert ^{2}\\
 & \geq\left(c_{1}k^{2/r}-c_{2}\right)\left\Vert s\right\Vert ^{2}
\end{align*}
from \prettyref{thm:First eigenvalue Bochner}, \prettyref{eq:Lichnerowicz for.}
and \prettyref{eq: semipositivity curvature estimate}. 
\end{proof}
We now derive as a corollary a spectral gap property for Kodaira Dirac
and Laplace operators $D_{k}$, $\Box_{k}$ corresponding to \prettyref{thm:First eigenvalue Bochner}. 
\begin{cor}
\label{cor: spectral gap Dirac} Under the hypotheses of Proposition
\prettyref{prop:Dk spectral estimate} there exist constants $c_{1},c_{2}>0$,
such that $\textrm{Spec}\left(\Box_{k}\right)\subset\left\{ 0\right\} \cup\left[c_{1}k^{2/r}-c_{2},\infty\right)$
for each $k$. Moreover, $\ker D_{k}^{-}=0$ and $H^{1}\left(Y;F\otimes L^{k}\right)=0$
for $k$ sufficiently large. 
\end{cor}

\begin{proof}
From Proposition \prettyref{prop:Dk spectral estimate}, it is clear
that 
\begin{equation}
\textrm{Spec}\left(\Box_{k}^{1}\right)\subset\left[c_{1}k^{2/r}-c_{2},\infty\right)\label{eq:spectral gap Kodaira}
\end{equation}
for some $c_{1},c_{2}>0$ giving the second part of the corollary.
Moreover, the eigenspaces of $\left.D_{k}^{2}\right|_{\Omega^{0,0/1}}$
with non-zero eigenvalue being isomorphic by Mckean-Singer, the first
part also follows. 
\end{proof}
The vanishing $H^{1}\left(Y;F\otimes L^{k}\right)=0$ for $k$ sufficiently
large also gives
\begin{align}
\textrm{dim }H^{0}\left(Y;F\otimes L^{k}\right)=\chi\left(Y;F\otimes L^{k}\right) & =k\left[\textrm{rk}\left(F\right)\int_{Y}c_{1}\left(L\right)\right]+\int_{Y}c_{1}\left(F\right)+1-g,\label{eq: hol. Euler characteristic}
\end{align}
by Riemann-Roch, with $\chi\left(Y;F\otimes L^{k}\right)$, $ch\left(F\otimes L^{k}\right)$,
$\textrm{Td}\left(Y\right)$, $g$ denoting the holomorphic Euler
characteristic, Chern character, Todd genus and genus of $Y$ respectively. 
\begin{rem}
\label{rem:good remark} Our proof of the last two results Proposition
\prettyref{prop:Dk spectral estimate} and Corollary \prettyref{cor: spectral gap Dirac}
follows \cite{Guillemin-Uribe88,MaMarinescu2002} from the positive
case. In the semipositive case however the proof only works on a Riemann
surface, since in higher dimensions there are more components to the
$\left[2\omega\left(R^{L}\right)-\tau^{L}\right]$ term \prettyref{eq: semipositivity curvature estimate}
in the Lichnerowicz formula \prettyref{eq:Lichnerowicz for.} which
semipositivity is insufficient to control. 

Indeed, Donnelly has shown a counterexample to the existence of a
spectral gap for semipositive line bundles in higher dimensions \cite{Donnelly2003}.
In the same paper \cite[Cor. 3.3]{Donnelly2003}, Donnelly has also
observed that on a Riemann surface the Kodaira Laplacian satisfies
an $O\left(1\right)$ spectral gap: $\textrm{Spec}\left(\Box_{k}\right)\subset\left\{ 0\right\} \cup\left[c,\infty\right)$,
for some $c>0$, by using its equivalence with the closed range hypothesis
for the Kohn Laplacian $\Box_{b}$ on the unit circle. It is however
crucial for our proof of the Bergman kernel expansion \prettyref{thm:Bergman kernel expansion}
that the size of the spectral gap is $O\left(k^{2/r}\right)$, as
in Corollary \prettyref{cor: spectral gap Dirac}, or that it is at
least growing with $k$.
\end{rem}

\subsection{\label{subsec:Bergman-kernel exp.}Bergman kernel expansion}

We now investigate the asymptotics of the Bergman kernel. Recall that
this is the Schwartz kernel $\Pi_{k}\left(y_{1},y_{2}\right)$ of
the projector onto the kernel of $\Box_{k}$ with respect to the $L^{2}$
inner product given by the metrics $g^{TY}$, $h^{F}$ and $h^{L}$.
Alternately, if $\left\{ s_{j}\right\} _{j=1}^{N_{k}^{0}}$ denotes
an orthonormal basis of eigensections of $H^{0}\left(X;F\otimes L^{k}\right)$
then 
\begin{equation}
\Pi_{k}\left(y_{1},y_{2}\right)\coloneqq\sum_{j=1}^{N_{k}}s_{j}\left(y_{1}\right)\otimes s_{j}\left(y_{2}\right)^{*}.\label{eq: Bergman kernel}
\end{equation}
We wish to describe the asymptotics of $\Pi_{k}\left(y,y\right)$
along the diagonal in $Y\times Y$.

Next as in \prettyref{subsec:First-eigenvalue expansion}, we fix
a geodesic coordinate system centered at $y\in Y$. By using parallel
transport of an orthonormal basis $\left\{ w\right\} $, $\left\{ l\right\} $,
$\left\{ f_{j}\right\} _{j=1}^{\textrm{rk}\left(F\right)}$ for $T_{y}^{1,0}Y$,
$L_{y}$ , $F$ with respect to $\nabla^{T^{1,0}Y}$, $\nabla^{L}$,
$\nabla^{F}$ respectively we obtain a local orthonormal trivialization
for the corresponding bundles over a geodesic ball $B_{2\varrho}\left(y\right)$.
In this frame and coordinate system, the connection on the tensor
product $\nabla^{\Lambda^{0,*}\otimes F\otimes L^{k}}$ again has
a similar expression as \prettyref{eq:connection tensor product}. 

We now define a modified frame $\left\{ \tilde{w}\right\} $ on $T^{1,0}\mathbb{R}^{2}$
which agrees with $\left\{ w\right\} $ on $B_{\varrho}\left(y\right)$
and with $\left\{ \frac{1}{\sqrt{2}}\left(\partial_{x_{1}}+i\partial_{x_{2}}\right)\right\} $
outside $B_{2\varrho}\left(y\right)$. Also define the modified metric
$\tilde{g}^{TY}$ and almost complex structure $\tilde{J}$ on $\mathbb{R}^{2}$
to be standard in this frame and hence agreeing with $g^{TY}$, $J$
on $B_{\varrho}\left(y\right)$. The Christoffel symbol of the corresponding
modified induced connection on $\Lambda^{0,*}$now satisfies $\tilde{a}^{\Lambda^{0,*}}=0$
outside $B_{2\varrho}\left(y\right)$.

Further we may as before define the modified connections $\tilde{\nabla}^{F},$
$\tilde{\nabla}^{L}$ \prettyref{eq:modified connection-1} as well
as the corresponding tensor product connection $\tilde{\nabla}^{\Lambda^{0,*}\otimes F\otimes L^{k}}$
which agrees with $\nabla^{\Lambda^{0,*}\otimes F\otimes L^{k}}$
on $B_{\varrho}\left(y\right)$. Clearly the curvature of the modified
connection $\tilde{\nabla}^{L}$ is given by $\tilde{R}^{L}$ \prettyref{eq:modified connection-1}.
This modified curvature can also be seen to be semipositive and vanishing
to order $r_{y}-2$ for $\varrho$ sufficiently small. We now define
the modified Kodaira Dirac operator on $\mathbb{R}^{2}$ by the similar
formula $\tilde{D}_{k}=c\circ\tilde{\nabla}^{\Lambda^{0,*}\otimes F\otimes L^{k}}$,
agreeing with $D_{k}$ on $B_{\varrho}\left(y\right)$ by \prettyref{eq:spin-c Dirac}.
This has a similar Lichnerowicz formula

\begin{align}
\tilde{D}_{k}^{2}=2\tilde{\Box}_{k}\coloneqq & \underbrace{\left(\tilde{\nabla}^{\Lambda^{0,*}\otimes F\otimes L^{k}}\right)^{*}\tilde{\nabla}^{\Lambda^{0,*}\otimes F\otimes L^{k}}}_{=\tilde{\Delta}_{k}}+k\left[2\omega\left(\tilde{R}^{L}\right)-\tilde{\tau}^{L}\right]+\left[2\omega\left(\tilde{R}^{F}\right)-\tilde{\tau}^{F}\right]+\frac{1}{2}\omega\left(\tilde{\kappa}\right)\label{eq:model Laplace}
\end{align}
with the adjoint being taken with respect to the metric $\tilde{g}^{TY}$
and corresponding volume form. Also the endomorphisms $\tilde{R}^{F},\tilde{\tau}^{F},\tilde{\tau}^{L}$
and $\omega\left(\tilde{\kappa}\right)$ are the obvious modifications
of \prettyref{eq:formulas Clifford-1} defined using the curvatures
of $\tilde{\nabla}^{F},\tilde{\nabla}^{L}$ and $\tilde{g}^{TY}$
respectively. The above \prettyref{eq:model Laplace} again agrees
with $\Box_{k}$ on $B_{\varrho}\left(y\right)$ while the endomorphisms
$\tilde{R}^{F},\tilde{\tau}^{F},\omega\left(\tilde{\kappa}\right)$
all vanish outside $B_{\varrho}\left(y\right)$. Being semi-bounded
below \prettyref{eq:model Laplace} is essentially self-adjoint. A
similar argument as Corollary \prettyref{cor: spectral gap Dirac}
gives a spectral gap 
\begin{equation}
\textrm{Spec}\left(\tilde{\Box}_{k}\right)\subset\left\{ 0\right\} \cup\left[c_{1}k^{2/r_{y}}-c_{2},\infty\right).\label{eq: spectral gap}
\end{equation}
To repeat some parts of the argument, first note that by construction
the localized Bochner Laplacian in \prettyref{eq:model Laplace} is
the rescaled model Laplacian $\tilde{\Delta}_{k}=k^{2/r}\mathscr{R}\Delta_{0}\mathscr{R}^{-1}$
on the complement of a compact ball $B_{\varrho}\left(y\right)^{c}$.
Using the global subelliptic estimate for the model Laplacian \prettyref{eq:model subelliptic estimate},
one obtains $c_{1},c_{2}>0$ such that 
\begin{equation}
\left\langle \tilde{\Delta}_{k}s,s\right\rangle \geq\left(c_{1}k^{2/r}-c_{2}\right)\left\Vert s\right\Vert ^{2}\label{eq:localized spectral gap}
\end{equation}
for each $s\in C_{c}^{\infty}\left(B_{\varrho}\left(y\right)^{c}\right)$
supported outside the ball. Combinig this with the local subelliptic
estimate on the compact ball $B_{\varrho}\left(y\right)$ one obtains
\prettyref{eq:localized spectral gap} for each $s\in C_{c}^{\infty}\left(\mathbb{R}^{2}\right)$
and hence for all $s$ in $\textrm{Dom}\left(\tilde{\Delta}_{k}\right)$
as an unbouded operator on $L^{2}$. The spectral gap \prettyref{eq: spectral gap}
for $\tilde{\Box}_{k}$ now again follows by the Lichnerwicz formula
as in the proof of \prettyref{cor: spectral gap Dirac}.

By elliptic regularity, the projector $\tilde{\Pi}_{k}$ from $L^{2}\left(\mathbb{R}^{2};\Lambda_{y}^{0,*}\otimes F_{y}\otimes L_{y}^{\otimes k}\right)$
onto $\ker\left(\tilde{\Box}_{k}\right)$ then has a smooth Schwartz
kernel with respect to the Riemannian volume of $\tilde{g}^{TY}$.

We are now ready to prove the Bergman kernel expansion \prettyref{thm:Bergman kernel expansion}. 
\begin{proof}[Proof of \prettyref{thm:Bergman kernel expansion}]
\label{bergman exp proof} First choose $\varphi\in\mathcal{S}\left(\mathbb{R}_{s}\right)$
satisfying $\varphi\left(0\right)=1$. For $c>0$, set $\varphi_{1}\left(s\right)=1_{\left[c,\infty\right)}\left(s\right)\varphi\left(s\right)$.
On account of the spectral gap Corollary \prettyref{cor: spectral gap Dirac},
and as $\varphi_{1}$ decays at infinity, we have 
\begin{align}
\varphi\left(\Box_{k}\right)-\Pi_{k} & =\varphi_{1}\left(\Box_{k}\right)\quad\textrm{ with }\nonumber \\
\left\Vert \Box_{k}^{a}\varphi_{1}\left(\Box_{k}\right)\right\Vert _{L^{2}\rightarrow L^{2}} & =O\left(k^{-\infty}\right)\label{eq: function approx. projector}
\end{align}
$\forall a\in\mathbb{N}$. Combining the above with semiclassical
Sobolev and elliptic estimates gives 
\begin{equation}
\left|\varphi\left(\Box_{k}\right)-\Pi_{k}\right|_{C^{l}\left(Y\times Y\right)}=O\left(k^{-\infty}\right),\label{eq:bergman vs schw.}
\end{equation}
$\forall l\in\mathbb{N}_{0}$. Next, we may write $\varphi\left(\Box_{k}\right)=\frac{1}{2\pi}\int_{\mathbb{R}}e^{i\xi\Box_{k}}\hat{\varphi}\left(\xi\right)d\xi$
via Fourier inversion. Since $\Box_{k}=\tilde{\Box}_{k}$ on $B_{\varrho}\left(y\right)$,
we may use a finite propagation argument as in \prettyref{eq: heat localization on diagonal}
to conclude 
\[
\varphi\left(\Box_{k}\right)\left(y',y\right)=\varphi\left(\tilde{\Box}_{k}\right)\left(y',0\right)+O\left(k^{-\infty}\right)
\]
for $y'\in B_{\frac{\varrho}{2}}\left(y\right)$. Next, since the
operator $\tilde{\Box}_{k}$ also satisfies a spectral gap \prettyref{eq: spectral gap},
similar arguments as above for the localized Kodaira Laplacian give
$\left\Vert \tilde{\Box}_{k}^{a}\varphi_{1}\left(\tilde{\Box}_{k}\right)\right\Vert _{L_{\textrm{loc}}^{2}\rightarrow L_{\textrm{loc}}^{2}}=O\left(k^{-\infty}\right)$.
And there after local elliptic regularity gives $\left|\varphi\left(\tilde{\Box}_{k}\right)-\tilde{\Pi}_{k}\right|_{C_{\textrm{loc}}^{l}\left(Y\times Y\right)}=O\left(k^{-\infty}\right)$
similar to \prettyref{eq:bergman vs schw.}. Thus we now have a localization
of the Bergman kernel 
\begin{align}
\Pi_{k}\left(.,y\right) & =O\left(k^{-\infty}\right),\quad\textrm{ on }B_{\varrho}\left(y\right)^{c}\nonumber \\
\Pi_{k}\left(.,y\right)-\tilde{\Pi}_{k}\left(.,0\right) & =O\left(k^{-\infty}\right),\quad\textrm{ on }B_{\varrho}\left(y\right).\label{eq:Bergman localization}
\end{align}
It thus suffices to consider the Bergman kernel of the model Kodaira
Laplacian \prettyref{eq:model Laplace} on $\mathbb{R}^{2}$.

Next, with the dilation $\delta_{k^{-1/r}}y=\left(k^{-1/r}y_{1},\ldots,k^{-1/r}y_{n-1}\right)$
as in Section \prettyref{subsec:First-eigenvalue expansion}, the
rescaled Kodaira Laplacian 
\begin{equation}
\boxdot\coloneqq k^{-2/r_{y}}\left(\delta_{k^{-1/r}}\right)_{*}\tilde{\Box}_{k}\label{eq:rescaled Dirac}
\end{equation}
satisfies 
\begin{align}
\varphi\left(\frac{\tilde{\Box}_{k}}{k^{2/r_{y}}}\right)\left(y,y'\right) & =k^{2/r_{y}}\varphi\left(\boxdot\right)\left(yk^{1/r_{y}},y'k^{1/r_{y}}\right)\label{eq: rescaling Schw kernel}
\end{align}
for $\varphi\in C_{c}^{\infty}\left(\mathbb{R}\right)$ as in \prettyref{eq: matrix coefficient in terms of model}.
Using a Taylor expansion via \prettyref{eq:modified connection-1},
the rescaled Kodaira Laplacian again has an expansion 
\begin{eqnarray}
\boxdot & = & \left(\sum_{j=0}^{N}k^{-j/r_{y}}\boxdot_{j}\right)+k^{-2\left(N+1\right)/r_{y}}\mathrm{E}_{N+1},\;\forall N,\label{eq: Taylor expansion Dirac}\\
\textrm{where each }\;\boxdot_{j} & = & a_{j;pq}\left(y\right)\partial_{y_{p}}\partial_{y_{q}}+b_{j;p}\left(y\right)\partial_{y_{p}}+c_{j}\left(y\right)\label{eq:operators in expansion}
\end{eqnarray}
is a $k$-independent self-adjoint, second-order differential operator
while each 
\begin{equation}
\mathrm{E}_{j}=\sum_{\left|\alpha\right|=N+1}y^{\alpha}\left[a_{j;pq}^{\alpha}\left(y;k\right)\partial_{y_{p}}\partial_{y_{q}}+b_{j;p}^{\alpha}\left(y;k\right)\partial_{y_{p}}+c_{j}^{\alpha}\left(y;k\right)\right]\label{eq: error operators}
\end{equation}
is a $k$-dependent self-adjoint, second-order differential operator
on $\mathbb{R}^{2}$ . Furthermore the functions appearing in \prettyref{eq:operators in expansion}
are again polynomials with degrees satisfying the same conditions
in \prettyref{eq:inequalities and restrictions on coefficients} and
\prettyref{eq:number of derivatives in coefficients}. While the coefficients
$a_{j;pq}^{\alpha}\left(y;k\right),b_{j;p}^{\alpha}\left(y;k\right),c_{j}^{\alpha}\left(y;k\right)$
of \prettyref{eq: error operators} are uniformly $C^{\infty}$ bounded
in $k$. Using \prettyref{eq:modified connection-1}, \prettyref{eq:model Bochner in coordinates},
\prettyref{eq: model Kodaira Laplace} and \prettyref{eq:model Lichnerowicz}
the leading term of \prettyref{eq: Taylor expansion Dirac} is computed
\begin{equation}
\boxdot_{0}=\Box_{y}\coloneqq\Box_{g^{TY},j_{y}^{r_{y}-2}R^{L},J^{TY}}\label{eq:leading term}
\end{equation}
in terms of the model Kodaira Laplacian on the tangent space $TY$
\prettyref{eq: model Kodaira Laplace}.

Next, we obtain an expansion for the right-hand side of \prettyref{eq: rescaling Schw kernel}
by a resolvent expansion for $\boxdot$ as in \prettyref{eq:Bochner functional expansion}.
Namely, we let $I_{j}\coloneqq\left\{ p=\left(p_{0},p_{1},\ldots\right)|p_{\alpha}\in\mathbb{N},\sum p_{\alpha}=j\right\} $
denote the set of partitions of the integer $j$ and define 
\begin{equation}
\mathtt{C}_{j}^{z}=\sum_{p\in I_{j}}\left(z-\boxdot_{0}\right)^{-1}\left[\Pi_{\alpha}\left[\boxdot_{p_{\alpha}}\left(z-\boxdot_{0}\right)^{-1}\right]\right].\label{eq: jth term kernel expansion}
\end{equation}
Then by repeated applications of the local elliptic estimate using
\prettyref{eq: Taylor expansion Dirac} we have 
\begin{equation}
\left(z-\mathrm{\boxdot}\right)^{-1}-\left(\sum_{j=0}^{N}k^{-j/r_{y}}\mathtt{C}_{j}^{z}\right)=O_{H_{\textrm{loc}}^{s}\rightarrow H_{\textrm{loc}}^{s+2}}\left(k^{-\left(N+1\right)/r_{y}}\left|\textrm{Im}z\right|^{-2Nr_{y}-2}\right),\label{eq: resolvent expansion}
\end{equation}
for each $N\in\mathbb{N},\,s\in\mathbb{R}$. Plugging the above into
the Helffer-Sjöstrand formula gives 
\begin{equation}
\varphi\left(\boxdot\right)-\sum_{j=0}^{N}k^{-j/r_{y}}\mathtt{C}_{j}^{\varphi}=O_{H_{\textrm{loc}}^{s}\rightarrow H_{\textrm{loc}}^{s+2}}\left(k^{-\left(N+1\right)/r_{y}}\right)\label{eq:functions Kodaira expansion}
\end{equation}
$\forall l,N\in\mathbb{N}_{0}$ and for some $k$-independent $\mathtt{C}_{j}^{\varphi}\in C^{\infty}\left(\mathbb{R}^{2}\times\mathbb{R}^{2}\right)$,
$j=0,1,\ldots$. The leading term $\mathtt{C}_{0}^{\varphi}=\varphi\left(\boxdot_{0}\right)=\varphi\left(\Box_{y}\right)$
is given in terms of the modal Kodaira Laplacian. Again a similar
argument as \prettyref{eq: regularized expansion}, replacing \prettyref{eq: resolvent expansion}
by the resolvent expansion for $\left(\mathrm{\boxdot}+1\right)^{-M}\left(z-\mathrm{\boxdot}\right)^{-1}$,
shows that \prettyref{eq:functions Kodaira expansion} is valid at
the level of kernels in $C^{l}\left(\mathbb{R}^{2}\times\mathbb{R}^{2}\right)$,
$\forall l\in\mathbb{N}$. Finally choosing $\varphi$ supported,
and equal to one, near $0$ gives \prettyref{eq:Bergmankernelexpansion}
from the spectral gap \prettyref{eq: spectral gap} as well as \prettyref{eq: rescaling Schw kernel}
and \prettyref{eq:functions Kodaira expansion}. The leading term
$c_{0}\left(y\right)=\mathtt{C}_{0}\left(0,0\right)=\Pi_{0}\left(0,0\right)\coloneqq\Pi^{\boxdot_{0}}\left(0,0\right)$
is seen to be the model Bergman kernel on the tangent space. See the
argument in \prettyref{sec:Model-operators} at the bottom of page
29 for the positivity $c_{0}\left(y\right)>0$. From this identification
of $c_{0}\left(y\right)$ with the model kernel one sees that it has
a locally smooth extension $c_{0,r_{y}}\left(y'\right)$ for $y'$
near $y$, depending only on the type $r_{y}$ at $y$. However such
an extension might have nothing to do with the Bergman kernel at points
$y'$ other than $y$. Finally, to show that there are no odd powers
of $k^{-j/r_{y}}$, one again notes that the operators $\boxdot_{j}$
\prettyref{eq:operators in expansion} change sign by $\left(-1\right)^{j}$
under $\delta_{-1}x\coloneqq-x$. Thus the Schwartz kernel for $\mathtt{C}_{j}^{z}$
\prettyref{eq: jth term kernel expansion} changes sign by $\left(-1\right)^{j}$
giving $\mathtt{C}_{j}^{z}\left(0,0\right)=0$ for $j$ odd. 
\end{proof}
We end by giving an example where semipositive bundles arise and where
the first term of the Bergman kernel expansion \prettyref{eq:Bergmankernelexpansion}
above can further be made explicit. 
\begin{example}
\label{exa:(Branched-coverings)}(Branched coverings) Consider $f:Y\rightarrow Y_{0}$
a branched covering of a Riemann surface $Y_{0}$ with branch points
$\left\{ y_{1},\ldots,y_{M}\right\} \subset Y$. The Hermitian holomorphic
line bundle on $Y$ is pulled back $\left(L,h^{L}\right)=\left(f^{*}L_{0},f^{*}h^{L_{0}}\right)$
from one on $Y_{0}$. If $\left(L_{0},h^{L_{0}}\right)$ is assumed
positive, then $\left(L,h^{L}\right)$ is semipositive with curvature
vanishing at the branch points. In particular, near a branch point
$y\in Y$ of local degree $\frac{r}{2}$ one may find holomorphic
geodesic coordinate such that the curvature is given by $R^{L}=\frac{r^{2}}{4}\left(z\bar{z}\right)^{r/2-1}R_{f\left(y\right)}^{L_{0}}+O\left(\left|z\right|^{r-1}\right)$.
The leading term of \prettyref{eq:Bergmankernelexpansion} is given
by the model Bergman kernel $\Pi^{\boxdot_{0}}\left(0,0\right)$ of
the operator $\boxdot_{0}=bb^{\dagger}$, $b^{\dagger}=2\partial_{\bar{z}}+a$,
$a=\frac{r}{4}z\left(z\bar{z}\right)^{r/2-1}R_{f\left(y\right)}^{L_{0}}$.
An orthonormal basis for $\textrm{ker}\left(\boxdot_{0}\right)$ is
then seen to be
\begin{align*}
s_{\alpha}\left(z\right) & \coloneqq\left(\frac{1}{2\pi}\frac{r}{\Gamma\left(\frac{2\left(\alpha+1\right)}{r}\right)}\left[R_{f\left(y\right)}^{L_{0}}\right]^{\frac{2\left(\alpha+1\right)}{r}}\right)^{1/2}z^{\alpha}e^{-\Phi},\quad\alpha\in\mathbb{N}_{0},\quad\textrm{with}\\
\Phi\left(z\right) & \coloneqq\frac{1}{4}\left(z\bar{z}\right)^{r/2}R_{f\left(y\right)}^{L_{0}}.
\end{align*}
Since $s_{\alpha}\left(0\right)=0$ for $\alpha\geq1$, the model
Bergman kernel at the origin is now computed
\[
c_{0}\left(y\right)=\Pi_{0}\left(0,0\right)\coloneqq\Pi^{\boxdot_{0}}\left(0,0\right)=\left|s_{0}\left(0\right)\right|^{2}=\frac{1}{2\pi}\frac{r}{\Gamma\left(\frac{2}{r}\right)}\left[R_{f\left(y\right)}^{L_{0}}\right]^{\frac{2}{r}}
\]
at the vanishing or branch point $y$.
\end{example}

\appendix

\section{\label{sec:Model-operators}Model operators}

Here we define certain model Bochner, Kodaira Laplacians and Dirac
operators acting on a vector space $V$. The Bochner Laplacian is
intrinsically associated to a triple $\left(V,g^{V},R^{V}\right)$
consisting of a metric $g^{V}$ and a non-vanishing tensor $0\neq R^{V}\in S^{r-2}V^{*}\otimes\Lambda^{2}V^{*}$,
$r\geq2$. While the Kodaira Laplacian depends on an additional complex
structure $J^{V}$ on $V$. Throughout the article, the vector space
$V=T_{y}Y$ is the tangent space of the manifold, with $g^{V}=g_{y}^{TY}$
the Riemannian metric, $J^{V}=J_{y}$ the complex structure and $R^{V}=j^{r-2}R_{y}^{L}$
the first non-vanishing jet of the auxiliary curvature $R^{L}$ at
a point $y\in Y$.

We say that the tensor $R^{V}$ is nondegenerate if the following
is satisfied
\begin{equation}
S^{r-s-2}V^{*}\otimes\Lambda^{2}V^{*}\ni i_{v}^{s}\left(R^{V}\right)=0,\,\forall s\leq r-2\,\implies T_{y}Y\ni v=0.\label{eq:non-degeneracy condition}
\end{equation}
Above $i^{s}$ denotes the $s$-fold contraction of the symmetric
part of $R^{V}$.

For $v_{1}\in V$, $v_{2}\in T_{v_{1}}V=V$, contraction of the antisymmetric
part, denoted by $\iota$, of $R^{V}$ gives $\iota_{v_{2}}R^{V}\in S^{r-2}V^{*}\otimes V^{*}$.
The contraction may then be evaluated $\left(\iota_{v_{2}}R^{V}\right)\left(v_{1}^{r}\right)$
at $v_{1}^{r}\coloneqq v_{1}^{\odot r}\otimes v_{1}$ for $v_{1}\in V$
and hence viewed as a homogeneous degree $r-1$ polynomial function
on $V$. The tensor $R^{V}$ now determines a one form $a^{R^{V}}\in\Omega^{1}\left(V\right)$
via 
\begin{equation}
a_{v_{1}}^{R^{V}}\left(v_{2}\right)\coloneqq\int_{0}^{1}d\rho\left(\iota_{v_{2}}R^{V}\right)\left(\left(\rho v_{1}\right)^{r}\right)=\frac{1}{r}\left(\iota_{v_{2}}R^{V}\right)\left(v_{1}^{r}\right),\label{eq:connection from tensot}
\end{equation}
which we may view as a unitary connection $\nabla^{R^{V}}=d+ia^{R^{V}}$
on a trivial Hermitian vector bundle $E$ of arbitrary rank over $V$.
The curvature of this connection is clearly $R^{V}$ now viewed as
a homogeneous degree $r-2$ polynomial function on $V$ valued in
$\Lambda^{2}V^{*}$. One now defines the model Bochner Laplacian,
intrinsically associated to the tuple $\left(V,g^{V},R^{V}\right)$,
via 
\begin{equation}
\Delta_{0}=\Delta_{g^{V},R^{V}}\coloneqq\left(\nabla^{R^{V}}\right)^{*}\nabla^{R^{V}}:C^{\infty}\left(V;E\right)\rightarrow C^{\infty}\left(V;E\right).\label{eq:model Bochner Laplacian}
\end{equation}
depending on the pair $\left(g^{V},R^{V}\right)$. An orthonormal
basis $\left\{ e_{1},e_{2},\ldots,e_{n}\right\} $, determines components
$R_{pq,\alpha}\coloneqq R^{V}\left(e^{\odot\alpha};e_{p},e_{q}\right)$,
$\alpha\in\mathbb{N}_{0}^{n-1}$, $\left|\alpha\right|=r-2$, as well
as linear coordinates $\left(y_{1},\ldots,y_{n}\right)$ on $V$.
The connection form in these coordinates is given by $a_{p}^{R^{V}}=\frac{i}{r}y^{q}y^{\alpha}R_{pq,\alpha}.$
While the model Laplacian \prettyref{eq:model Bochner Laplacian}
is given 
\begin{equation}
\Delta_{0}=-\sum_{q=1}^{n}\left(\partial_{y_{p}}+\frac{i}{r}y^{q}y^{\alpha}R_{pq,\alpha}\right)^{2}.\label{eq:model Bochner in coordinates}
\end{equation}
As in \prettyref{eq: Fourier decomposition Laplacian}, the above
may now be related to the nilpotent sR Laplacian on the product $S_{\theta}^{1}\times V$
given by 
\begin{equation}
\hat{\Delta}_{0}=\hat{\Delta}_{g^{V},R^{V}}\coloneqq-\sum_{q=1}^{n}\left(\partial_{y_{p}}+\frac{i}{r}y^{q}y^{\alpha}R_{pq,\alpha}\partial_{\theta}\right)^{2},\label{eq:nilpotent Laplacian}
\end{equation}
and corresponding to the sR structure $\left(S_{\theta}^{1}\times V,\ker\left(d\theta+a^{R^{V}}\right),\text{\ensuremath{\pi}}^{*}g^{V},\,d\theta\textrm{vol}g^{V}\right)$
where the sR metric corresponds to $g^{V}$ under the natural projection
$\pi:S_{\theta}^{1}\times V\rightarrow V$. Note that the above differs
from the usual nilpotent approximation of the sR Laplacian since it
acts on the product with $S^{1}$. Following \cite[Part III, Sec. 16]{Rothschild-Stein76},
the above satisfies a subelliptic estimate: there exists $C>0$ such
that 
\begin{equation}
\left\Vert s\right\Vert _{H^{1/r}}^{2}\leq C\left[\left\langle \hat{\Delta}_{0}s,s\right\rangle +\left\Vert s\right\Vert _{L^{2}}^{2}\right],\quad\forall s\in C_{c}^{\infty}\left(S_{\theta}^{1}\times V\right).\label{eq:model subelliptic estimate}
\end{equation}
As \prettyref{eq:Fourier mode heat kernel}, the heat kernels of \prettyref{eq:model Bochner Laplacian},
\prettyref{eq:nilpotent Laplacian} are now related 
\begin{equation}
e^{-t\Delta_{0}}\left(y,y'\right)=\int e^{-i\theta}e^{-t\hat{\Delta}_{0}}\left(y,0;y',\theta\right)d\theta.\label{eq: model heat ker. reln.}
\end{equation}

Next, assume that the vector space $V$ is of even dimension and additionally
is equipped with an orthogonal endomorphism $J^{V}\in O\left(V\right)$;
$\left(J^{V}\right)^{2}=-1$. This gives rise to a linear integrable
almost complex structure on $V$ , a decomposition $V\otimes\mathbb{C}=V^{1,0}\oplus V^{0,1}$
into $\pm i$ eigenspaces of $J$ and a Clifford multiplication endomorphism
$c:V\rightarrow\textrm{End}\left(\Lambda^{*}V^{0,1}\right)$. We further
assume that $R^{V}$ is a $\left(1,1\right)$ form with respect to
$J$ that is $S^{k}V^{*}\ni R^{V}\left(w_{1},w_{2}\right)=0$, $\forall w_{1},w_{2}\in V^{1,0}$.
The $\left(0,1\right)$ part of the connection form \prettyref{eq:connection from tensot}
then gives a holomorphic structure on the trivial Hermitian line bundle
$\mathbb{C}$ with holomorphic derivative $\bar{\partial}_{\mathbb{C}}=\bar{\partial}+\left(a^{V}\right)^{0,1}$.
One may now define the model Kodaira Dirac and Laplace operators,
intrinsically associated to the tuple $\left(V,g^{V},R^{V},J^{V}\right)$,
via 
\begin{align}
D_{0}=D_{g^{V},R^{V},J^{V}}\coloneqq & \sqrt{2}\left(\bar{\partial}_{\mathbb{C}}+\bar{\partial}_{\mathbb{C}}^{*}\right)\label{eq:model Kodaira Dirac}\\
\Box_{0}=\Box_{g^{V},R^{V},J^{V}}\coloneqq & \frac{1}{2}\left(D_{g^{V},R^{V},J^{V}}\right)^{2}\label{eq: model Kodaira Laplace}
\end{align}
acting on $C^{\infty}\left(V;\Lambda^{*}V^{0,1}\right)$. The above
\prettyref{eq:model Bochner Laplacian}, \prettyref{eq: model Kodaira Laplace}
are related by the Lichnerowicz formula 
\begin{equation}
\Box_{0}=\Delta_{0}+c\left(R^{V}\right)\label{eq:model Lichnerowicz}
\end{equation}
where $c\left(R^{V}\right)=\sum_{p<q}R_{pq}^{i_{1}\ldots i_{r-2}}y_{i_{1}\ldots}y_{i_{r-2}}c\left(e_{p}\right)c\left(e_{q}\right)$.
We may choose a complex orthonormal basis $\left\{ w_{j}\right\} _{j=1}^{m}$
of $V^{1,0}$ that diagonalizes the tensor $R^{V}$: $R^{V}\left(w_{i},\bar{w}_{j}\right)=\delta_{ij}R_{j\bar{j}}$;
$R_{i\bar{j}}\in S^{r-2}V^{*}$. This gives complex coordinates on
$V$ in which \prettyref{eq: model Kodaira Laplace} may be written
as 
\begin{align}
\Box_{0} & =\sum_{q=1}^{\textrm{dim}V/2}b_{j}b_{j}^{\dagger}+2\left(\partial_{z_{j}}a_{j}+\partial_{\bar{z}_{j}}\bar{a}_{j}\right)\bar{w}_{j}i_{\bar{w}_{j}}\label{eq:model Kodaira Lap.}\\
\textrm{where }\quad b_{j}\coloneqq-2\partial_{z_{j}}+\bar{a}_{j},\quad & b_{j}^{\dagger}=2\partial_{\bar{z}_{j}}+a_{j},\quad\textrm{ for }\quad a_{j}=\frac{1}{r}R_{j\bar{j}}z_{j},\nonumber 
\end{align}
with each $R_{j\bar{j}}\left(z\right)$, $1\leq j\leq\textrm{dim}V/2$,
being a real homogeneous function of order $r-2$.

Being symmetric with respect to the standard Euclidean density and
semi-bounded below, both $\Delta_{0}$ and $\Box_{0}$ are essentially
self-adjoint on $L^{2}$. We shall need the following information
regarding their spectrum. 
\begin{prop}
\label{prop:model spectra} There exists $c>0$ such that $\textrm{Spec}\left(\Delta_{0}\right)\subset\left[c,\infty\right)$.
For $R^{V}$ satisfying the non-degeneracy condition \prettyref{eq:non-degeneracy condition}
one has $\textrm{EssSpec}\left(\Delta_{0}\right)=\emptyset$. While
for $\textrm{dim}V=2$ with $R^{V}\left(w,\bar{w}\right)\geq0$ semipositive
one has $\textrm{Spec}\left(\Box_{0}\right)\subset\left\{ 0\right\} \cup\left[c,\infty\right).$ 
\end{prop}

\begin{proof}
The proof of the first part is similar to that of \prettyref{eq:asymptotic bounds for l0}.
Introduce the deformed Laplacian $\Delta_{k}\coloneqq\Delta_{g^{V},kR^{V}}$
obtained by rescaling the tensor $R^{V}$. From \prettyref{eq:model Bochner in coordinates}
$\Delta_{k}=k^{2/r}\mathscr{R}\Delta_{0}\mathscr{R}^{-1}$ are conjugate
under the rescaling $\mathscr{R}:C^{\infty}\left(V;E\right)\rightarrow C^{\infty}\left(V;E\right)$,
$\left(\mathscr{R}u\right)\left(x\right)\coloneqq u\left(yk^{1/r}\right)$
implying 
\begin{align}
\textrm{Spec}\left(\Delta_{k}\right) & =k^{2/r}\textrm{Spec}\left(\Delta_{0}\right)\nonumber \\
\textrm{EssSpec}\left(\Delta_{k}\right) & =k^{2/r}\textrm{EssSpec}\left(\Delta_{0}\right)\label{eq: rescaling of spectra}
\end{align}
By an argument similar to \prettyref{eq:asymptotic bounds for l0},
one has $\textrm{Spec}\left(\Delta_{k}\right)\subset\left[c_{1}k^{2/r}-c_{2},\infty\right)$
for some $c_{1},c_{2}>0$ for $R^{V}\neq0$. From here $\textrm{Spec}\left(\Delta_{0}\right)\subset\left[c,\infty\right)$
follows. Next, under the non-degeneracy condition, the order of vanishing
of the homogeneous curvature $R^{V}$ (of the homogeneous connection
$a^{R^{V}}$\prettyref{eq:connection from tensot}) is seen to be
maximal at the origin: $\textrm{ord}_{y}\left(R^{V}\right)<r-2$ for
$y\neq0$. Following a similar sub-elliptic estimate \prettyref{eq:local subelliptic estimate}
on $V\times S_{\theta}^{1}$ as in \prettyref{eq:asymptotic bounds for l0},
we have 
\[
k^{2/\left(r-1\right)}\left\Vert u\right\Vert ^{2}\leq C\left[\left\langle \Delta_{k}u,u\right\rangle +\left\Vert u\right\Vert _{L^{2}}^{2}\right],\quad\forall u\in C_{c}^{\infty}\left(V\setminus B_{1}\left(0\right)\right),
\]
holds on the complement of the unit ball centered at the origin. Combining
the above with Persson's characterization of the essential spectrum
\cite[Ch. 3]{Agmon-book-1982} 
\[
\textrm{EssSpec}\left(\Delta_{k}\right)=\sup_{R}\inf_{\substack{\left\Vert u\right\Vert =1\\
u\in C_{c}^{\infty}\left(V\setminus B_{R}\left(0\right)\right)
}
}\left\langle \Delta_{k}u,u\right\rangle ,
\]
we have $\textrm{EssSpec}\left(\Delta_{k}\right)\subset\left[c_{1}k^{2/\left(r-1\right)}-c_{2},\infty\right)$.
From here and using \prettyref{eq: rescaling of spectra}, $\textrm{EssSpec}\left(\Delta_{0}\right)=\emptyset$
follows.

For the final part, similarly set $\Box_{k}\coloneqq\Box_{g^{V},kR^{V},J^{V}}$
and note that $k^{2/r}\textrm{Spec}\left(\Box_{0}\right)=\textrm{Spec}\left(\Box_{k}\right)\subset\left\{ 0\right\} \cup\left[c_{1}k^{2/r}-c_{2},\infty\right)$
by an argument similar to Corollary \prettyref{cor: spectral gap Dirac}. 
\end{proof}
Next, the heat $e^{-t\Delta_{0}}$, $e^{-t\Box_{0}}$ and wave $e^{it\sqrt{\Delta_{0}}}$,
$e^{it\sqrt{\Box_{0}}}$ operators being well-defined by functional
calculus, a finite propagation type argument as in \prettyref{eq:heat exponential decay off-diagonal}
gives $\varphi\left(\Delta_{0}\right)\left(.,0\right)\in\mathcal{S}\left(V\right)$,
$\varphi\left(\Box_{0}\right)\left(.,0\right)\in\mathcal{S}\left(V\right)$
are of Schwartz class for $\varphi\in\mathcal{S}\left(\mathbb{R}\right)$.
When $\textrm{EssSpec}\left(\Delta_{0}\right)=\emptyset$ any eigenfunction
of $\Delta_{0}$ also lies in $\mathcal{S}\left(V\right)$. Finally,
under the hypothesis of \prettyref{prop:model spectra}, the Schwartz
kernel $\Pi_{0}\left(.,0\right)\in\mathcal{S}\left(V\right)$ of the
projector $\Pi_{0}=\Pi^{\Box_{0}}$ onto the kernel of $\Box_{0}$
is also seen to be of Schwartz class, on choosing $\varphi$ supported
close to the origin. 

The constant $a_{0}\coloneqq\Pi_{0}\left(0,0\right)$ is also the
leading term in the boundary expansion $\Pi_{D}\left(z,z\right)\sim a_{0}\left(-\rho\right)^{-2-\frac{2}{r}}$
for the Bergman kernel of a weakly pseudoconvex finite type domain
$D\coloneqq\left\{ \rho<0\right\} \subset\mathbb{C}^{2}$ as $z\rightarrow x'\in\partial D$
a point on its boundary \cite[Thm. 2]{HsiaoSavale-2022}. Here $r=r\left(x'\right)$
is the typer of the point on the boundary. In this case, \cite[Thm. 2]{Catlin89}
proved the lower bound $\Pi_{D}\left(z,z\right)\geq c\left(-\rho\right)^{-2-\frac{2}{r}}$
for some $c>0$. Thus $a_{0}>0$.

We now state another proposition regarding the heat kernel of $\Delta_{0}$.
Below we denote $\lambda_{0}\left(\Delta_{0}\right)\coloneqq\textrm{inf Spec}\left(\Delta_{0}\right)$
. 
\begin{prop}
\label{prop:large time heat =00003D000026 bottom of spec}For each
$\varepsilon>0$ there exist $t,R>0$ such that the integral of the
heat kernel satisfies
\[
\frac{\int_{B_{R}\left(0\right)}dx\left[\Delta_{0}e^{-t\Delta_{0}}\right]\left(x,x\right)}{\int_{B_{R}\left(0\right)}dx\,e^{-t\Delta_{0}}\left(x,x\right)}\leq\lambda_{0}\left(\Delta_{0}\right)+\varepsilon
\]
\end{prop}

\begin{proof}
Setting $P\coloneqq\Delta_{0}-\lambda_{0}\left(\Delta_{0}\right)$
it suffices to show 
\[
\frac{\int_{B_{R}\left(0\right)}dx\left[Pe^{-tP}\right]\left(x,x\right)}{\int_{B_{R}\left(0\right)}dx\,e^{-tP}\left(x,x\right)}\leq\varepsilon
\]
for some $t,R>0$. With $\Pi_{\left[0,x\right]}^{P}$ denoting the
spectral projector onto $\left[0,x\right]$, we split the numerator
\[
\int_{B_{R}\left(0\right)}dx\left[Pe^{-tP}\right]\left(x,x\right)=\int_{B_{R}\left(0\right)}dx\left[\Pi_{\left[0,4\varepsilon\right]}^{P}Pe^{-tP}\right]\left(x,x\right)+\int_{B_{R}\left(0\right)}dx\left[\left(1-\Pi_{\left[0,4\varepsilon\right]}^{P}\right)Pe^{-tP}\right]\left(x,x\right).
\]
From $P\geq0$, $\Pi_{\left[0,4\varepsilon\right]}^{P}Pe^{-tP}\leq4\varepsilon e^{-tP}$
and $\left(1-\Pi_{\left[0,4\varepsilon\right]}^{P}\right)Pe^{-tP}\leq ce^{-3\varepsilon t}$,
$\forall t\geq1$, we may bound 
\begin{equation}
\frac{\int_{B_{R}\left(0\right)}dx\left[Pe^{-tP}\right]\left(x,x\right)}{\int_{B_{R}\left(0\right)}dx\,e^{-tP}\left(x,x\right)}\leq4\varepsilon+\frac{ce^{-3\varepsilon t}R^{n-1}}{\int_{B_{R}\left(0\right)}dx\,e^{-tP}\left(x,x\right)}\label{eq:breakup est.}
\end{equation}
$\forall R,t\geq1$. Next, as $0\in\textrm{Spec}\left(P\right)$ there
exists $\left\Vert \psi_{\varepsilon}\right\Vert _{L^{2}}=1$, $\left\Vert P\psi_{\varepsilon}\right\Vert _{L^{2}}\leq\varepsilon$.
It now follows that $\left\Vert \psi_{\varepsilon}-\Pi_{\left[0,2\varepsilon\right]}^{P}\psi_{\varepsilon}\right\Vert \leq\frac{1}{2}$
and hence 
\[
\begin{split}\frac{1}{2} & =-\frac{1}{4}+\int_{B_{R_{\varepsilon}}\left(0\right)}dx\left|\psi_{\varepsilon}\left(x\right)\right|^{2}\leq\int_{B_{R_{\varepsilon}}\left(0\right)}dx\left|\int dy\Pi_{\left[0,2\varepsilon\right]}^{P}\left(x,y\right)\psi_{\varepsilon}\left(y\right)\right|^{2}\\
 & \leq\int_{B_{R_{\varepsilon}}\left(0\right)}dx\left(\int dy\Pi_{\left[0,2\varepsilon\right]}^{P}\left(x,y\right)\Pi_{\left[0,2\varepsilon\right]}^{P}\left(y,x\right)\right)=\int_{B_{R_{\varepsilon}}\left(0\right)}dx\Pi_{\left[0,2\varepsilon\right]}^{P}\left(x,x\right),
\end{split}
\]
for some $R_{\varepsilon}>0$, using $\left(\Pi_{\left[0,2\varepsilon\right]}^{P}\right)^{2}=\Pi_{\left[0,2\varepsilon\right]}^{P}$
and Cauchy-Schwartz. This gives 
\[
\int_{B_{R_{\varepsilon}}\left(0\right)}dx\,e^{-tP}\left(x,x\right)\geq\frac{e^{-2\varepsilon t}}{2}\,,\quad t>1.
\]
Plugging this last inequality into \prettyref{eq:breakup est.} gives
\[
\frac{\int_{B_{R_{\varepsilon}}\left(0\right)}dx\left[Pe^{-tP}\right]\left(x,x\right)}{\int_{B_{R_{\varepsilon}}\left(0\right)}dx\,e^{-tP}\left(x,x\right)}\leq4\varepsilon+ce^{-\varepsilon t}R_{\varepsilon}^{n-1}
\]
from which the theorem follows on choosing $t$ large. 
\end{proof}

\section*{Declarations}

\subsection*{Conflict of interest }

On behalf of all authors, the corresponding author states that there
is no conflict of interest.

\subsection*{Data sharing }

Data sharing not applicable to this article as no datasets were generated
or analysed during the current study.

\bibliographystyle{siam}
\bibliography{biblio}

\end{document}